\documentclass[11pt, draft]{amsart}
\usepackage{amssymb, amstext, amscd, amsmath, amssymb}
\usepackage{mathtools, xypic, paralist, color, dsfont, rotating}
\usepackage{verbatim}
\usepackage{enumerate,enumitem}
\usepackage{float}
\usepackage{setspace}
\usepackage{pifont}
\usepackage{tikz-cd}

\usepackage{marginnote}
\usepackage[a4paper]{geometry}
\floatstyle{boxed} 
\restylefloat{figure}

\numberwithin{equation}{section}
\let\OLDthebibliography\thebibliography
\renewcommand\thebibliography[1]{
  \OLDthebibliography{#1}
  \setlength{\parskip}{0pt}
  \setlength{\itemsep}{2pt plus 0.5ex}
}

\makeatletter
\def\@cite#1#2{{\m@th\upshape\bfseries%
[{#1\if@tempswa{\m@th\upshape\mdseries, #2}\fi}]}}
\makeatother

\theoremstyle{plain}
\newtheorem{theorem}{Theorem}[subsection]
\newtheorem{corollary}[theorem]{Corollary}
\newtheorem{proposition}[theorem]{Proposition}
\newtheorem{lemma}[theorem]{Lemma}

\theoremstyle{definition}
\newtheorem{definition}[theorem]{Definition}

\newtheorem{remark}[theorem]{Remark}

\theoremstyle{remark}

\mathtoolsset{centercolon}

\newcommand{\B}{{\mathcal{B}}}

\newcommand{\I}{{\mathcal{I}}}
\newcommand{\J}{{\mathcal{J}}}
\newcommand{\K}{{\mathcal{K}}}
\renewcommand{\L}{{\mathcal{L}}}
\newcommand{\M}{{\mathcal{M}}}
\newcommand{\N}{{\mathcal{N}}}
\renewcommand{\O}{{\mathcal{O}}}

\newcommand{\T}{{\mathcal{T}}}

\def\al{\alpha}

\def\ga{\gamma}

\def\de{\delta}
\def\ze{\zeta}

\def\la{\lambda}
\def\La{\Lambda}

\newcommand{\bC}{\mathbb{C}}

\newcommand{\bI}{\mathbb{I}}

\newcommand{\bN}{\mathbb{N}}
\newcommand{\bT}{\mathbb{T}}
\newcommand{\bZ}{\mathbb{Z}}

\newcommand{\fJ}{{\mathfrak{J}}}

\newcommand{\foral}{\text{ for all }}
\newcommand{\qand}{\quad\text{and}\quad}

\newcommand{\ca}{\mathrm{C}^*}

\newcommand{\ol}{\overline}

\newcommand{\End}{\operatorname{End}}

\newcommand{\id}{{\operatorname{id}}}

\newcommand{\inv}{{\operatorname{inv}}}

\newcommand{\mt}{\emptyset}

\newcommand{\spn}{\operatorname{span}}

\newcommand{\supp}{\operatorname{supp}}

\newcommand{\sca}[1]{\left\langle#1\right\rangle} 
\newcommand{\nor}[1]{\left\Vert #1\right\Vert} 
\newcommand{\un}[1]{{\underline{#1}}} 

\addtocontents{toc}{\protect\setcounter{tocdepth}{1}}

\begin{document}

\title[Gauge-invariant ideal structure]{Gauge-invariant ideal structure of C*-algebras associated with proper product systems over $\bZ_+^d$}


\author[J.A. Dessi]{Joseph A. Dessi}
\email{joseph.dessi1@gmail.com}



\begin{abstract}
We show that the gauge-invariant ideal parametrisation results of the author and Kakariadis are in agreement with those of Bilich in the case of a proper product system over $\bZ_+^d$.
This is accomplished in two ways: first via the use of Nica-covariant representations and Gauge-Invariant Uniqueness Theorems (the indirect route), and second via the definitions of the parametrising objects alone (the direct route).
We then apply our findings to simplify the main parametrisation result of the author and Kakariadis in the proper case, thereby fully describing the gauge-invariant ideal structure of each equivariant quotient of the Toeplitz-Nica-Pimsner algebra.
We close by providing applications in the contexts of C*-dynamical systems and row-finite higher-rank graphs.
\end{abstract}

\maketitle

\tableofcontents

\section{Introduction}\label{S:intro}

\subsection{Background}\label{Ss:background}

A prominent feature of the theory of operator algebras is the quantisation procedure by which a geometric/topological object can be studied via bounded linear operators on a Hilbert space.
The goal is to associate such an object with a C*-algebra in a rigid way, such that properties of the original structure are reflected by properties of the C*-algebra (and vice versa). 
In this way, the powerful and well-developed theory of C*-algebras can be brought to bear on the study of other mathematical structures.
In recent years, there has been interest in encoding this procedure in a uniform way, i.e., accounting for a multitude of examples via a single framework. 

A contemporary tool in this endeavour is that of product systems, whose associated C*-algebras account for a vast array of C*-constructions associated with a unital subsemigroup $P$ of a discrete group $G$.
Structures encompassed by this language include (but are not limited to) C*-dynamical systems, higher-rank graphs and subshifts. 
A pertinent feature of product systems is their ability to encode transformations that may not be reversible, and as such the associated C*-algebras provide an ample source of examples and counterexamples.
In turn, there is motivation to analyse the structure of these C*-algebras, and interpret the results with respect to the applications that the product system construction affords.
Much progress has been made in this direction in the case of $P=\bZ_+$; however, the situation changes when we consider more general semigroups.
There are many open questions even in the case of $P=\bZ_+^d$.

The case of $P=\bZ_+$ is the case of a single C*-correspondence $X$, the study of whose C*-algebras was contextualised by Pimsner \cite{Pim95}.
The quantisation is implemented via a Fock space construction, in which the elements of $X$ are treated as left creation operators.
These operators, together with the coefficient algebra of $X$ (viewed as a family of operators itself), give rise to the Toeplitz-Pimsner algebra $\T_X$.
Of particular interest is a specific equivariant quotient (i.e., a quotient by a gauge-invariant ideal): the Cuntz-Pimsner algebra $\O_X$.
The latter is the minimal C*-algebra that contains an isometric copy of $X$, and it is this boundary behaviour that allows for the recovery of numerous (rank-one) C*-constructions. 
The C*-crossed product induced by a single $*$-automorphism and the Cuntz-Krieger algebra associated with a row-finite directed graph, for example, are both incarnations of the C*-algebra $\O_X$.

In light of the array of applications, C*-algebras associated with C*-correspondences have been explored in detail.
Important developments in this direction include the study of ideal structure and simplicity \cite{CKO19}, K-theory computation \cite{Kat04} and classification \cite{BTZ18}, necessary and sufficient conditions for nuclearity and exactness \cite{Kat04}, the decomposition and parametrisation of the KMS-simplex \cite{Kak20,LN04} and, of particular importance to the current work, the parametrisation of gauge-invariant ideals \cite{Kat07}.
Focusing on the latter, the parametrisation is implemented by pairs of ideals of the coefficient algebra satisfying conditions related to the underlying C*-correspondence.
If the C*-correspondence is induced by a geometric/topological object, then this description can be translated directly in terms of properties of the inducing object.
For example, the gauge-invariant ideals of the Cuntz-Krieger algebra of a row-finite directed graph are in bijection with the hereditary saturated vertex sets of the graph, in accordance with \cite{BPRS00}.

Moving beyond $\bZ_+$, many of the aforementioned results do not have clear extensions to the general case.
However, by imposing additional structure on the product system $X$, progress can be made.
One such addition is compact alignment for product systems over quasi-lattices, as pioneered by Fowler \cite{Fow02}.
We can also ask that the representations of $X$ preserve compact alignment, leading to the notion of Nica-covariant representations.
In this case the associated C*-algebras admit a Wick ordering due to the Nica-covariant relations of the Fock representation, allowing for a tractable analysis via cores.
The KMS-simplex of the Fock C*-algebra and particularly KMS-states of finite type have been studied by Afsar, Larsen and Neshveyev \cite{ALN20}, unifying multiple works (see also \cite{Ch20} for the case of higher-rank graphs and \cite{Kak20-2} for finite-rank product systems).
We can still make sense of compact alignment when extending to product systems over right LCM semigroups, and a thorough study of the associated C*-algebras was provided by Kwa\'sniewski and Larsen \cite{KL19a, KL19b}.
A key difference compared to the rank-one case is that the Fock C*-algebra is not universal for all representations, in general.
However, we do have that the Fock C*-algebra is universal for all Nica-covariant representations when $X$ is compactly aligned over a unital right LCM subsemigroup of an amenable discrete group (note that $P=\bZ_+^d$ resides within this framework).
In their recent work, Brix, Carlsen and Sims \cite{BCS23} explore the ideal structure of C*-algebras related to commuting local homeomorphisms, pushing the theory beyond simplicity.

Until recently, the problem of ascertaining the appropriate Cuntz-type object for product systems has been open.
Work in this direction commenced with the results of Fowler \cite{Fow02}. 
Sims and Yeend \cite{SY10} provided an answer in the case of compactly aligned product systems over quasi-lattices, and showed that this C*-algebra (referred to as the Cuntz-Nica-Pimsner algebra) accounts for numerous examples.
Co-universality of the Cuntz-Nica-Pimsner algebra (under an appropriate amenability assumption) was clarified by Carlsen, Larsen, Sims and Vittadello \cite{CLSV11}.
The appropriate Cuntz-type object for compactly aligned product systems over right LCM semigroups was identified as the C*-envelope of the (nonselfadjoint) tensor algebra (equipped with the natural coaction) by Dor-On, Kakariadis, Katsoulis, Laca and Li \cite{DKKLL20}.
Nuclearity and exactness was addressed by Kakariadis, Katsoulis, Laca and Li \cite{KKLL21b}.
The complete picture was provided in the general case by Sehnem \cite{Seh18, Seh21} via strong covariance relations, linking the Cuntz-type object with the C*-envelope of the tensor algebra.

The preceding results fall into the broader programme of bringing C*-algebras of product systems into the remit of Elliott's Classification Programme.
A key result in this direction for $P=\bZ_+$ has been provided by Brown, Tikuisis and Zelenberg \cite{BTZ18}, wherein a sufficient condition for classifiability of the Cuntz-Pimsner algebra in terms of properties of the C*-correspondence and its coefficient algebra is provided.
A corresponding result for the Cuntz-Nica-Pimsner algebra in higher-rank cases has not yet been achieved.
Indeed, one of the key advantages of the rank-one case is that the strong covariance relations defining the Cuntz-Pimsner algebra are simple and algebraic, induced by a single ideal of the coefficient algebra introduced by Katsura \cite{Kat03}.
In the general case the picture is significantly more complicated, since the strong covariance relations may not adopt the simple algebraic format of the rank-one case.
For example, the relations defining the Cuntz-Nica-Pimsner algebra of Sims and Yeend \cite{SY10} are based on families of compact operators induced by all possible finite subsets of the underlying semigroup.

\subsection{Motivation}\label{Ss:mot}

Let $X$ be a compactly aligned product system over the semigroup $P=\bZ_+^d$ that additionally satisfies the strong compact alignment condition of \cite[Definition 2.2]{DK18}.
This condition, introduced by Dor-On and Kakariadis \cite{DK18}, is advantageous because it ensures that the strong covariance relations defining the Cuntz-Nica-Pimsner algebra are simple and algebraic in format and are induced by a family of $2^d-1$ ideals of the coefficient algebra (or $2^d$ ideals if we count the trivial relations induced by the zero ideal).
This picture is in analogy with the rank-one case, opening a direction for lifting results from this setting.

Strong compact alignment proved to be the linchpin that enabled a parametrisation of the gauge-invariant ideals of the Toeplitz-Nica-Pimsner algebra $\N\T_X$ (i.e., the universal C*-algebra for the Nica-covariant representations of $X$) via certain $2^d$-tuples of ideals of the coefficient algebra, as established by the author and Kakariadis \cite[Theorem 4.2.3]{DeK24}.
The introduction and study of these $2^d$-tuples, termed \emph{NT-$2^d$-tuples} \cite[Definition 4.1.4]{DeK24}, is the focus of the aforementioned work.
In particular, describing NT-$2^d$-tuples via product system operations alone is a key point of attention.
In \cite{DeK24} we also prove a Gauge-Invariant Uniqueness Theorem (with another obtained as a subcase) \cite[Theorems 3.2.11 and 3.4.9]{DeK24}; we parametrise the gauge-invariant ideals of every equivariant quotient of $\N\T_X$ (e.g., the Cuntz-Nica-Pimsner algebra $\N\O_X$) \cite[Theorem 4.2.11]{DeK24}; we identify the lattice operations rendering the parametrisations lattice isomorphisms \cite[Propositions 4.2.6, 4.2.7 and 4.2.10]{DeK24}, and we interpret the parametrising objects in the contexts of regular product systems \cite[Corollary 5.2.3]{DeK24}, C*-dynamical systems \cite[Corollary 5.3.5]{DeK24}, higher-rank graphs \cite[Corollary 5.4.14]{DeK24}, and product systems over $\bZ_+^d$ in which each fibre (except the coefficient algebra) admits a finite frame \cite[Corollary 5.5.23]{DeK24}. 

Strong compactly aligned product systems include as a subclass the proper product systems over $\bZ_+^d$, i.e., those product systems over $\bZ_+^d$ whose left actions are by compact operators.
These product systems account for numerous important examples, including C*-dynamical systems and row-finite higher-rank graphs.
Fixing such a product system $X$, a parametrisation of the gauge-invariant ideals of $\N\T_X$ was provided by Bilich \cite[Theorem 4.15 (1)]{BB24}, contemporaneously with \cite{DeK24}.
This parametrisation is also implemented via certain $2^d$-tuples of ideals of the coefficient algebra, though they are defined differently to the NT-$2^d$-tuples.
More specifically, these $2^d$-tuples are termed \emph{T-families} \cite[Definition 4.2]{BB24}, and their introduction/study is a key aspect of the aforementioned work.
In \cite{BB24} Bilich also proves a Gauge-Invariant Uniqueness Theorem \cite[Corollary 4.14]{BB24}, parametrises the gauge-invariant ideals of $\N\O_X$ \cite[Theorem 4.15 (2)]{BB24}, and interprets the main results in the context of row-finite higher-rank graphs \cite[Theorem 5.5]{BB24}.

It is important to note that the methods employed in \cite{BB24,DeK24} differ substantially.
In brief, the former argues using product system extensions while the latter proceeds by analysing maximal families.
Nevertheless, the parametrisation results of \cite{BB24,DeK24} share key commonalities, including the format of the parametrising objects (namely $2^d$-tuples of the coefficient algebra satisfying certain properties) and the use of a Gauge-Invariant Uniqueness Theorem as an important part of the proofs.
As such, it is natural to ask the following: ``if we restrict to the setting of a proper product system $X$ over $\bZ_+^d$, are the NT-$2^d$-tuples of $X$ exactly the T-families of $X$?"
Providing an affirmative answer to this question, as well as illuminating the connections between \cite{BB24} and \cite{DeK24}, are the main points of motivation for the current work.

Equipped with this affirmative answer, we will use it to provide a complete and succinct description of the gauge-invariant ideal structure of every equivariant quotient of $\N\T_X$, simplifying \cite[Theorem 4.2.11]{DeK24} in the proper setting.
We will then interpret the modified parametrising objects in the contexts of C*-dynamical systems and row-finite higher-rank graphs, simplifying \cite[Corollary 5.3.5]{DeK24} and part of \cite[Corollary 5.4.14]{DeK24}, as well as demonstrating an alignment with \cite[Theorem 5.5]{BB24}.

\subsection{Description of main results}\label{Ss:resultsumm}

Let us fix notation (see in conjunction with the general notation that we adopt in Subsections \ref{Ss:notation} and \ref{Ss:prod sys}).
We write $[d] := \{1, \dots, d\}$ for $d \in \bN$.
We write $\un{n}$ for the elements of $\bZ_+^d$ and will denote its generators by $\un{i}$ for $i \in [d]$.
We write $\un{n} \perp F$ for $F \subseteq [d]$ if $\supp \un{n} \cap F = \mt$.
Moreover, we write $\un{1}_F:= \sum\{\un{i}\mid i\in F\}$ for $\mt \neq F \subseteq [d]$.

Throughout the subsection, we will take $X = \{X_{\un{n}}\mid\un{n}\in\bZ_+^d\}$ to be a product system over $\bZ_+^d$ with coefficients in a C*-algebra $A$ that is also \emph{proper}, i.e., we have that $\phi_\un{n}(A)\subseteq\K(X_\un{n})$ for all $\un{n}\in\bZ_+^d$.
We work at this level of generality for the majority of the discussion, with some excursions to the more general setting of strong compactly aligned product systems where appropriate.
Fixing $\un{n}\in\bZ_+^d, \mt\neq F\subseteq[d]$ and an ideal $I$ of $A$, we write
\[
X_{\un{n}}^{-1}(I) := \{a \in A \mid \sca{X_{\un{n}}, a X_{\un{n}}} \subseteq I \}\qand J_F(I,X):=\{a\in A\mid aX_F^{-1}(I)\subseteq I\},
\]
where $X_F^{-1}(I):=\bigcap\{X_\un{m}^{-1}(I)\mid\un{0}\neq\un{m}\leq\un{1}_F\}$.
A \emph{$2^d$-tuple (of $X$)} is a family $\L:=\{\L_F\}_{F\subseteq[d]}$ of $2^d$ non-empty subsets of $A$.

If $(\pi,t)$ is a Nica-covariant representation of $X$ that acts on a Hilbert space $H$, we write $\psi_\un{n}$ for the induced $*$-representation of $\K(X_\un{n})$ for each $\un{n}\in\bZ_+^d$.
For $i \in [d]$, we use an approximate unit $(k_{\un{i},\la})_{\la\in\La}$ of $\K(X_{\un{i}})$ to define the projection $p_{\un{i}}:=\textup{w*-}\lim_\la \psi_{\un{i}}(k_{\un{i},\la})$, and we set
\begin{equation*}
q_\mt:= I_H
\text{ and }
q_F:=\prod_{i\in F}(I_H-p_{\un{i}}) \textup{ for $\mt\neq F\subseteq [d]$}.
\end{equation*}
Fixing $a\in A$ and $\mt\neq F\subseteq[d]$, the key relation is that 
\[
\pi(a) q_F = \pi(a) + \sum \{ (-1)^{|\un{n}|} \psi_{\un{n}}(\phi_{\un{n}}(a)) \mid \un{0} \neq \un{n} \leq \un{1}_F\}
\]
and thus $\pi(a) q_F \in \ca(\pi,t)$, although it may not be that $q_F \in \ca(\pi,t)$.
We reserve $(\ol{\pi}_X, \ol{t}_X)$ for the universal Nica-covariant representation of $X$.
Due to the aforementioned relation, each $2^d$-tuple $\L$ of $X$ induces a canonical gauge-invariant ideal $\sca{\ol{\pi}_X(\L_F) \ol{q}_{X,F} \mid F \subseteq [d]}$ of $\N\T_X$.
We write $\N\O(\L,X)$ for the corresponding equivariant quotient.

The main result in \cite{DK18} is that $\N\O_X \cong \N\O(\I,X)$ for the family $\I := \{\I_F\}_{F \subseteq [d]}$, where
\[ 
\I_F := \bigcap\{X_\un{n}^{-1}(\J_F)\mid \un{n}\perp F\}
\quad \text{for} \quad
\J_F :=
(\bigcap_{i\in F}\ker\phi_{\un{i}})^\perp.
\]
We note that $\I_\mt = \J_\mt = \{0\}$.
Every $\I_F$ is $F^\perp$-invariant (in fact the largest $F^\perp$-invariant ideal of $\J_F$), and the family $\I$ is partially ordered in the sense that $\I_{F_1}\subseteq\I_{F_2}$ whenever $F_1\subseteq F_2\subseteq[d]$.
We can abstract these properties as follows. 
Given a $2^d$-tuple $\L$ of $X$, we say that $\L$ is \emph{$X$-invariant} if $\left[\sca{X_\un{n},\L_F X_\un{n}}\right]\subseteq\L_F$ for all $\un{n}\perp F$ and $F\subseteq[d]$.
We say that $\L$ is \emph{partially ordered} if $\L_{F_1}\subseteq\L_{F_2}$ whenever $F_1\subseteq F_2\subseteq[d]$.
If $\L$ consists of ideals, then to each $\mt\neq F\subsetneq[d]$ we may associate the following two ideals of $A$:
\[
\L_{\inv, F} := \bigcap_{\un{n}\perp F}X_\un{n}^{-1}(\cap_{F\subsetneq D}\L_D)
\qand
\L_{\lim, F} := \{a\in A \mid \lim_{\un{n}\perp F}\|\phi_\un{n}(a)+\K(X_\un{n}\L_F)\|=0\}.
\]

A $2^d$-tuple $\L$ of $X$ is said to be an \emph{NT-$2^d$-tuple (of $X$)} if it satisfies the following four conditions:
\begin{enumerate}
\item $\L$ consists of ideals and $\L_F\subseteq J_F(\L_\mt,X)$ for all $\mt\neq F\subseteq[d]$,
\item $\L$ is $X$-invariant,
\item $\L$ is partially ordered, 
\item for each $\mt\neq F\subsetneq[d]$, we have that
\[
\bigg(\bigcap_{\un{n}\perp F}X_\un{n}^{-1}(J_F(\L_\mt,X))\bigg)\cap\L_{\inv,F}\cap\L_{\lim,F}\subseteq\L_F.
\]
\end{enumerate}
The NT-$2^d$-tuples of $X$ parametrise the gauge-invariant ideals of $\N\T_X$ \cite[Theorem 4.2.3]{DeK24}.
With minor adjustments, we obtain a parametrisation of the gauge-invariant ideals of $\N\O(\K,X)$ for an arbitrary $2^d$-tuple $\K$ of $X$.
More precisely, we say that an NT-$2^d$-tuple $\L$ of $X$ is a \emph{$\K$-relative NO-$2^d$-tuple (of $X$)} if $\K_F\subseteq\L_F$ for all $F\subseteq[d]$.
Such families parametrise the gauge-invariant ideals of $\N\O(\K,X)$ \cite[Theorem 4.2.11]{DeK24}.
We refer to the $\I$-relative NO-$2^d$-tuples of $X$ simply as \emph{NO-$2^d$-tuples (of $X$)} and note that these families parametrise the gauge-invariant ideals of $\N\O_X$ \cite[Corollary 4.2.12]{DeK24}.

A $2^d$-tuple $\L$ of $X$ is said to be a \emph{T-family (of $X$)} if it consists of ideals and satisfies
\[
\L_F=X_\un{i}^{-1}(\L_F)\cap\L_{F\cup\{i\}}\foral F\subsetneq[d], i\in[d]\setminus F.
\]
We say that a T-family $\L$ of $X$ is an \emph{O-family (of $X$)} if $\I_F\subseteq\L_F$ for all $F\subseteq[d]$.
The T-families (resp. O-families) of $X$ parametrise the gauge-invariant ideals of $\N\T_X$ (resp. $\N\O_X$) \cite[Theorem 4.15]{BB24}.

Ascertaining the alignment of NT-$2^d$-tuples with T-families is the central focus of the current work (the alignment of NO-$2^d$-tuples with O-families then follows immediately).
In Section \ref{S:BBDeK} we establish this alignment via an indirect route, exploiting Nica-covariant representations and Gauge-Invariant Uniqueness Theorems.
In doing so, we clarify several connections between \cite{BB24} and \cite{DeK24}.
In Section \ref{S:NT=T} we employ a more direct approach, demonstrating the alignment between NT-$2^d$-tuples and T-families using the definitions alone.
A strength of this methodology lies in the fact that it requires a minimal amount of background knowledge in product system theory.
This leads to our first main result.

\medskip

\noindent
{\bf Theorem A.}\emph{(Theorem \ref{T:NT=T})} Let $X$ be a proper product system over $\bZ_+^d$ with coefficients in a C*-algebra $A$.
Then the following hold:
\begin{enumerate}
\item the NT-$2^d$-tuples of $X$ are exactly the T-families of $X$, and
\item the NO-$2^d$-tuples of $X$ are exactly the O-families of $X$.
\end{enumerate}

\medskip

In Section \ref{S:app} we present several applications of Theorem A.
We commence by simplifying \cite[Theorem 4.2.11]{DeK24} in the proper case.
To this end, we argue that the $\K$-relative NO-$2^d$-tuples of $X$ are exactly the T-families $\L$ of $X$ that satisfy $\K_F\subseteq\L_F$ for all $F\subseteq[d]$.
We refer to such families as \emph{$\K$-relative O-families (of $X$)} and arrive at our next main result.

\medskip

\noindent
{\bf Theorem B.} \emph{(Theorem \ref{T:NOparamprop})} Let $X$ be a proper product system over $\bZ_+^d$ with coefficients in a C*-algebra $A$ and let $\K$ be a $2^d$-tuple of $X$.
Then there exists an order-preserving bijection between the set of $\K$-relative O-families of $X$ and the set of gauge-invariant ideals of $\N\O(\K,X)$.

\medskip

It should be noted that the lattice operations on the set of $\K$-relative O-families that bolster the bijection of Theorem B to a lattice isomorphism are clarified in \cite[Propositions 4.2.6, 4.2.7 and 4.2.10]{DeK24}.
Next we interpret the $\K$-relative O-families in the context of C*-dynamical systems, thereby simplifying \cite[Corollary 5.3.5]{DeK24}. 
Given a C*-dynamical system $(A,\al,\bZ_+^d)$, we write $X_\al$ for the associated proper product system.

\medskip

\noindent
{\bf Corollary C.} \emph{(Corollary \ref{C:dynsysinterp})} Let $(A,\al,\bZ_+^d)$ be a C*-dynamical system and let $\K$ and $\L$ be $2^d$-tuples of $X_\al$.
Then $\L$ is a $\K$-relative O-family of $X_\al$ if and only if the following three conditions hold:
\begin{enumerate}
\item $\L$ consists of ideals,
\item $\L_F=\al_\un{i}^{-1}(\L_F)\cap\L_{F\cup\{i\}}\foral F\subsetneq[d] \; \text{and} \; i\in[d]\setminus F$, and
\item $\K\subseteq\L$.
\end{enumerate}

\medskip

Finally, we interpret the $\K$-relative O-families in the context of row-finite higher-rank graphs, thereby simplifying the row-finite case of \cite[Corollary 5.4.14]{DeK24} and recovering the first part of \cite[Theorem 5.5]{BB24}.
Given a row-finite $k$-graph $(\La,d)$, we write $X(\La)$ for the associated proper product system.
Given a $2^d$-tuple $\L$ of ideals of $c_0(\La^\un{0})$, we write $H_\L$ for the associated family of vertex sets.

\medskip

\noindent
{\bf Corollary D.} \emph{(Corollary \ref{C:kgraphinterp})} Let $(\La,d)$ be a row-finite $k$-graph.
Let $\K$ and $\L$ be $2^k$-tuples of $X(\La)$ and suppose that $\K$ consists of ideals.
Then $\L$ is a $\K$-relative O-family of $X(\La)$ if and only if the following three conditions hold:
\begin{enumerate}
\item $\L$ consists of ideals,
\item $H_{\L,F}=\{v\in\La^\un{0}\mid s(v\La^\un{i})\subseteq H_{\L,F}\}\cap H_{\L,F\cup\{i\}}$ for all $F\subsetneq[k] \; \text{and} \; i\in[k]\setminus F$, and
\item $H_{\K,F}\subseteq H_{\L,F}$ for all $F\subseteq[k]$.
\end{enumerate}


\subsection{Contents of sections}\label{Ss:secsum}

In Section \ref{S:prod sys} we provide an exposition on the aspects of C*-correspondence and product system theory that we will need.
Upon collecting the requisite results concerning C*-correspondences, we present Katsura's parametrisation of gauge-invariant ideals \cite{Kat07} for comparison with the main results of \cite{BB24, DeK24}.
We then move on to consider product systems over $\bZ_+^d$.
We pay particular attention to the strong compactly aligned product systems and present the key results of \cite{DK18} that follow from the strong compact alignment condition.
We also define the main C*-algebras of interest, namely $\N\T_X$ and $\N\O_X$.
Finally, we outline the quotient product system construction, which is needed at several points in the sequel.

In Section \ref{S:BBDeK} we present the more specialised tools that are employed in \cite{BB24, DeK24}.
We start by working with respect to an arbitrary strong compactly aligned product system.
We define the relative Cuntz-Nica-Pimsner algebras $\N\O(\K,X)$ and an important ideal family arising from each Nica-covariant representation.
Building on this, we present a Gauge-Invariant Uniqueness Theorem for certain ``maximal" $2^d$-tuples.
We also define the crucial notions of invariance and partial ordering for $2^d$-tuples of non-empty subsets of the coefficient algebra, as well as the $\L^{(1)}$ construction- a key step in \cite{DeK24}.
We then define NT-$2^d$-tuples, NO-$2^d$-tuples and $\K$-relative NO-$2^d$-tuples, and give the main result of \cite{DeK24}.
Next, we restrict to the setting of proper product systems and introduce the key concepts/results of \cite{BB24}.
In particular, we define T-families and O-families, present a Gauge-Invariant Uniqueness Theorem for T-families, and give the main result of \cite{BB24}.
We then demonstrate the alignment of NT-$2^d$-tuples with T-families by taking a detour via Nica-covariant representations and the previously mentioned Gauge-Invariant Uniqueness Theorems.
Most of the material in this section is not new, and serves more as an abridged account of \cite{BB24, DeK24}. 
The reader is directed to Propositions \ref{P:maxNT}, \ref{P:repTfam} and \ref{P:Tfamga}, as well as to Lemma \ref{L:GIUTequiv} for the new results in this section.

In Section \ref{S:NT=T} we turn to showing the alignment of NT-$2^d$-tuples with T-families directly, using the definitions alone.
In particular, the direct passage from T-families to NT-$2^d$-tuples resolves an open question from the author's PhD thesis \cite[Remark 6.2.8]{De24}.

In Section \ref{S:app} we give several applications of the main result.
More specifically, we use it to simplify \cite[Theorem 4.2.11]{DeK24} in the proper case.
We then interpret the simplified parametrising objects in the contexts of C*-dynamical systems and row-finite higher-rank graphs.

\medskip

\noindent {\bf Acknowledgements.} Part of the material presented in the current work appears in the PhD thesis of the author.
Accordingly, the author acknowledges support from EPSRC as part of his PhD thesis on the programme ``The Structure of C*-Algebras of Product Systems" (Ref. 2441268).
The author also gives heartfelt thanks to Evgenios Kakariadis, for carefully reading the initial drafts of the manuscript and offering helpful feedback thereupon.

\medskip

\noindent {\bf Open access statement.} For the purpose of open access, the author has applied a Creative Commons Attribution (CC BY) license to any Author Accepted Manuscript (AAM) version arising.

\section{C*-correspondences and product systems}\label{S:prod sys}

\noindent We begin by presenting the key concepts from the theory of C*-correspondences and product systems that we will need.
The results in this section are stated without proof and not always at full generality, e.g., we only consider product systems over $\bZ_+^d$ rather than an arbitrary unital semigroup.
For a more comprehensive and general introduction, including full proofs, the reader is directed to \cite[Chapter 2]{De24}.

\subsection{Notation}\label{Ss:notation}

By a lattice we will always mean a distributive lattice with operations $\vee$ and $\wedge$.
We write $\bZ_+$ for the nonnegative integers $\{0,1,\dots\}$ and $\bN$ for the positive integers $\{1,2,\dots\}$.
We denote the unit circle in the complex plane by $\bT$.
We use $H$ to denote an arbitrary Hilbert space.
If $A,B$ and $C$ are sets and $f \colon A\times B\to C$ is a map, then we set
\[
f(A,B) := \{f(a,b)\mid a\in A,b\in B\};
\]
for example $\sca{H,H} := \{\sca{\xi,\eta}\mid\xi,\eta\in H\}$.
If $V$ is a normed vector space and $S\subseteq V$ is a subset, then $[S]$ denotes the norm-closed linear span of $S$ inside $V$.


All ideals of C*-algebras are taken to be two-sided and norm-closed.
If $A$ is a C*-algebra and $S\subseteq A$ is a subset, then $\ca(S)$ denotes the C*-subalgebra of $A$ generated by $S$, and $\sca{S}$ denotes the ideal of $A$ generated by $S$.
If $I\subseteq A$ is an ideal, then we set $I^\perp:=\{a\in A\mid aI=\{0\}\}$.
Let $A$ and $B$ be C*-algebras generated by subsets $\{a_i\mid i\in\bI\}$ and $\{b_i\mid i\in\bI\}$, respectively, where $\bI$ is a non-empty set.
Then a map $\Phi \colon A \to B$ is called \emph{canonical} if it preserves generators of the same index, i.e., $\Phi(a_i) = b_i$ for all $i \in \bI$.

\subsection{C*-correspondences}\label{Ss:C*-cor}

We assume familiarity with the elementary theory of right Hilbert C*-modules.
The reader is addressed to \cite{Lan95, MT05} for an excellent introduction to the subject.
We will briefly outline the fundamentals of the theory of C*-correspondences.
We also recount Katsura's parametrisation of gauge-invariant ideals \cite{Kat07}.

Let $A$ be a C*-algebra and $X$ be a right Hilbert $A$-module.
We write $\L(X)$ for the C*-algebra of adjointable operators on $X$, and $\K(X)$ for the ideal of (generalised) compact operators on $X$.
Recall that $\K(X)$ is densely spanned by the rank-one operators $\Theta_{\xi,\eta}^X \colon \zeta\mapsto\xi\sca{\eta,\zeta}$, for $\xi,\eta,\zeta\in X$.
Where there is no potential ambiguity, we will write $\Theta_{\xi,\eta}$ instead of $\Theta_{\xi, \eta}^X$.

A \emph{C*-correspondence} $X$ over a C*-algebra $A$ is a right Hilbert $A$-module equipped with a left action implemented by a $*$-homomorphism $\phi_X \colon A\to\L(X)$.
When the left action is clear from the context, we will abbreviate $\phi_X(a)\xi$ as $a\xi$, for $a\in A$ and $\xi\in X$.
We say that $X$ is \emph{non-degenerate} if $[\phi_X(A)X]=X$. 
If $\phi_X$ is injective, then we say that $X$ is \emph{injective}. 
If $\phi_X(A)\subseteq\K(X)$, then we say that $X$ is \emph{proper}.
If $X$ is injective and proper, then we say that $X$ is \emph{regular}.

Any C*-algebra $A$ can be viewed as a non-degenerate regular C*-correspondence over itself, with right (resp. left) action given by right (resp. left) multiplication in $A$, and $A$-valued inner product given by $\sca{a,b}=a^*b$ for all $a,b\in A$.
Then $A\cong\K(A)$ by the injective left action $\phi_A$, and non-degeneracy is deduced by applying an approximate unit.

Let $X$ and $Y$ be C*-correspondences over a C*-algebra $A$.
We call an $A$-bimodule linear map $u \colon X\to Y$ a \emph{unitary} if it is a surjection that preserves the $A$-valued inner product.
If such a unitary exists, then it is adjointable \cite[Theorem 3.5]{Lan95} and we say that $X$ and $Y$ are \emph{unitarily equivalent} (symb. $X\cong Y$).

We write $X\otimes_A Y$ for the $A$-balanced tensor product.
Given $S\in\L(X)$, there exists an operator $S\otimes\text{id}_Y\in\L(X\otimes_A Y)$ defined on simple tensors by $x\otimes y\mapsto (Sx)\otimes y$ for all $x\in X$ and $y\in Y$, e.g., \cite[p. 42]{Lan95}.
The assignment $S\mapsto S\otimes\text{id}_Y$ constitutes a unital $*$-homomorphism $\L(X)\to\L(X\otimes_A Y)$.
In this way we can define a left action $\phi_{X\otimes_A Y}$ on $X\otimes_A Y$ by 
\[
\phi_{X\otimes_A Y}(a)= \phi_X(a) \otimes\id_Y\foral a\in A, 
\]
thereby endowing $X\otimes_A Y$ with the structure of a C*-correspondence over $A$.
The $A$-balanced tensor product is associative.
Moreover, the right action of $X$ yields a unitary $X \otimes_A A \to X$ determined by $\xi \otimes a\mapsto \xi a$ for all $\xi\in X$ and $a\in A$.
The left action of $X$ yields a unitary $A\otimes_A X\to[\phi_X(A)X]$ determined by $a\otimes\xi\mapsto\phi_X(a)\xi$ for all $a\in A$ and $\xi\in X$.
We recall \cite[Lemma 4.6]{Lan95}, rewritten slightly to match our setting.

\begin{lemma}\label{L:lance} \cite[Lemma 4.6]{Lan95}
Let $X$ and $Y$ be C*-correspondences over a C*-algebra $A$.
For $x\in X$, the equation $\Theta_x(y)=x\otimes y \; (y\in Y)$ defines an element $\Theta_x\in\L(Y,X\otimes_AY)$ satisfying
\begin{align*}
\|\Theta_x\| & =\|\phi_Y(\sca{x,x}^{1/2})\|\leq\|x\| \qand
\Theta_x^*(x'\otimes y)  =\phi_Y(\sca{x,x'})y \; (x'\in X, y\in Y).
 \end{align*}
\end{lemma}

A \emph{(Toeplitz) representation} $(\pi,t)$ of the C*-correspondence $X$ on $\B(H)$ is a pair of a $*$-homomorphism $\pi \colon A\to \B(H)$ and a linear map $t \colon X\to \B(H)$ that preserves the left action and inner product of $X$.
Then $(\pi,t)$ automatically preserves the right action of $X$. 
There exists a $*$-homomorphism $\psi\colon\K(X)\to\B(H)$ such that $\psi(\Theta_{\xi,\eta})=t(\xi)t(\eta)^*$ for all $\xi,\eta\in X$, e.g., \cite[Proposition 4.6.3]{BO08}.
We say that $(\pi,t)$ is \emph{injective} if $\pi$ is injective; then both $t$ and $\psi$ are isometric.
We write $\ca(\pi,t)$ for the C*-algebra generated by $\pi(A)$ and $t(X)$.

We say that a representation $(\pi,t)$ of $X$ \emph{admits a gauge action} $\ga$ if there exists a family $\{\ga_z\}_{z\in\bT}$ of $*$-endomorphisms of $\ca(\pi,t)$ such that
\[ 
\ga_z(\pi(a))=\pi(a) \foral a\in A \; \text{and} \; \ga_z(t(\xi))=zt(\xi) \foral \xi\in X, 
\]
for each $z\in\bT$.
When such a gauge action $\ga$ exists, it is necessarily unique. 
We also have that each $\ga_z$ is a $*$-automorphism, the family $\{\ga_z\}_{z\in\bT}$ is point-norm continuous, and we obtain a group homomorphism (also denoted by $\ga$) defined by
\[
\ga\colon\bT\to\text{Aut}(\ca(\pi,t)); z\mapsto\ga_z\foral z\in\bT.
\]
An ideal $\mathfrak{J}\subseteq\ca(\pi,t)$ is called \emph{gauge-invariant} or \emph{equivariant} if $\ga_z(\mathfrak{J})\subseteq\mathfrak{J}$ for all $z\in\bT$ (and thus $\ga_z(\fJ) = \fJ$ for all $z \in \bT$).

The \emph{Toeplitz-Pimsner algebra} $\T_X$ is the universal C*-algebra with respect to the representations of $X$.
Let $J$ be a subset of $A$ satisfying $J\subseteq\phi_X^{-1}(\K(X))$. 
The \emph{$J$-relative Cuntz-Pimsner algebra} $\O(J,X)$ is the universal C*-algebra with respect to the representations $(\pi,t)$ of $X$ that satisfy $\pi(a) =\psi(\phi_X(a))$ for all $a\in J$.
Traditionally the relative Cuntz-Pimsner algebras are defined with respect to ideals of $A$ rather than just subsets, however the two versions are equivalent since $\O(J,X)=\O(\sca{J},X)$.
When $J=\{0\}$, we have that $\O(J,X)=\T_X$.
When we take $J$ to be the ideal
\[
J_X:=(\ker\phi_X)^\perp\cap\phi_X^{-1}(\K(X))\subseteq A,
\]
we obtain that $\O(J_X,X)$ is the \emph{Cuntz-Pimsner algebra} $\O_X$  \cite{Kat03}.
Katsura's ideal $J_X$ is the largest ideal to which the restriction of $\phi_X$ is injective with image contained in $\K(X)$ \cite{Kat03}.

One of the main tools in the theory is the Gauge-Invariant Uniqueness Theorem, obtained in its full generality by Katsura \cite{Kat07}.
An alternative proof can be found in \cite{Kak16}, and Frei \cite{Fre21} extended this method to include all relative Cuntz-Pimsner algebras, in connection with \cite{Kat07}.

\begin{theorem}\cite[Corollary 11.8]{Kat07} \label{T:giut}
Let $X$ be a C*-correspondence over a C*-algebra $A$, let $J\subseteq A$ be an ideal satisfying $J\subseteq J_X$ and let $(\pi,t)$ be a representation of $X$. 
Then we have that $\O(J,X)\cong\ca(\pi,t)$ via a (unique) canonical $*$-isomorphism if and only if $(\pi,t)$ is injective, admits a gauge action and satisfies $\pi^{-1}(\psi(\K(X)))=J$.
\end{theorem}

Let $A$ be a C*-algebra, let $I\subseteq A$ be an ideal and let $X$ be a right Hilbert $A$-module.
Then the set $XI$ is a closed linear subspace of $X$ that is invariant under the right action of $A$, e.g, \cite[p. 576]{FMR03} or \cite[Corollary 1.4]{Kat07}.
In particular, we have that $[XI]=XI$. 
Consequently, $XI$ is itself a right Hilbert $A$-module under the operations and $A$-valued inner product inherited from $X$. 
We may also view $XI$ as a right Hilbert $I$-module. 
Due to \cite[Lemma 2.6]{FMR03}, we will identify $\K(XI)$ as an ideal of $\K(X)$ in the following natural way:
\begin{align*}
\K(XI) & =\ol{\spn}\{\Theta_{\xi,\eta}^X \mid \xi,\eta\in XI\}\subseteq\K(X).
\end{align*}
When $X$ is in addition a C*-correspondence over $A$, we may equip $XI$ with a C*-correspondence structure via the left action
\[
\phi_{XI}\colon A\to\L(XI); \phi_{XI}(a)=\phi_X(a)|_{XI} \foral a \in A.
\]

Following \cite{Kat07}, and in order to ease notation, we will use the symbol $[ \hspace{1pt} \cdot \hspace{1pt} ]_I$ to denote the quotient maps associated with a right Hilbert $A$-module $X$ and an ideal $I\subseteq A$.
For example, we use it for both the quotient map $A\to A/I\equiv [A]_I$ and the quotient map $X\to X/XI\equiv[X]_I$.
We equip the complex vector space $[X]_I$ with the following right $[A]_I$-module multiplication:
\[ 
[\xi]_I[a]_I=[\xi a]_I \foral \xi\in X,a\in A, 
\]
as well as the following $[A]_I$-valued inner product:
\[ 
\sca{[\xi]_I,[\eta]_I}=[\sca{\xi,\eta}]_I \foral \xi,\eta\in X. 
\]
Consequently, the space $[X]_I$ carries the structure of an inner-product right $[A]_I$-module. 
By \cite[Lemma 1.5]{Kat07}, the canonical norm on $[X]_I$ induced by the $[A]_I$-valued inner product coincides with the usual quotient norm. 
Thus $[X]_I$ is a right Hilbert $[A]_I$-module. 
We may define a $*$-homomorphism $[\hspace{1pt} \cdot \hspace{1pt}]_I\colon \L(X)\to\L([X]_I)$ by
\[ 
[S]_I[\xi]_I=[S\xi ]_I \foral S\in\L(X), \xi\in X. 
\]
We include \cite[Lemma 1.6]{Kat07} in its entirety, as it will be especially relevant when we restrict our attention to proper product systems.

\begin{lemma}\label{L:Kat07}\cite[Lemma 1.6]{Kat07}
Let $X$ be a right Hilbert module over a C*-algebra $A$ and let $I\subseteq A$ be an ideal.
Then for all $\xi,\eta\in X$, we have that $[\Theta_{\xi,\eta}^X]_I=\Theta_{[\xi]_I,[\eta]_I}^{[X]_I}$.
The restriction of the map $[\hspace{1pt} \cdot \hspace{1pt}]_I \colon \L(X)\to\L([X]_I)$ to $\K(X)$ is a surjection onto $\K([X]_I)$ with kernel $\K(XI)$.
\end{lemma}

Therefore, given an ideal $I \subseteq A$, we obtain the surjective maps
\begin{align*}
A \to A/I & \textup{ with kernel $I$}, \\
X \to X/ XI & \textup{ with kernel $XI$}, \\
\K(X) \to \K(X/XI) & \textup{ with kernel $\K(XI)$},
\end{align*}
as well as the map $\L(X) \to \L(X/XI)$ (which may \emph{not} be surjective), all of which will be denoted by the same symbol $[\hspace{1pt} \cdot \hspace{1pt}]_I$.
Lemma \ref{L:Kat07} implies that if $k\in\K(X)$, then
\begin{equation}\label{Eq: comp}
k\in\K(XI) \iff \sca{X, k X} \subseteq I.
\end{equation}
Since $\K(XI)$ is an ideal in $\K(X)$ and $\K(X)$ is an ideal in $\L(X)$, we have that $\K(XI)$ is an ideal in $\L(X)$.
Hence we may consider the quotient C*-algebra $\L(X)/\K(XI)$.

If $X$ is a C*-correspondence over $A$, then we need to make an additional imposition on $I$ in order for $[X]_I$ to carry a canonical structure as a C*-correspondence over $[A]_I$. 
More specifically, we say that $I$ is \emph{positively invariant (for $X$)} if it satisfies
\[ 
[\sca{X,IX}]\subseteq I. 
\]
When $I$ is positively invariant, we can equip $[X]_I$ with the left action defined by
\[
\phi_{[X]_I} \colon [A]_I \to \L([X]_I); [a]_I \mapsto [\phi_X(a)]_I\foral a\in A.
\] 
To ease notation, we will denote $\phi_{[X]_I}$ by $[\phi_X]_I$.
Moreover, we define two ideals of $A$ that are related to $I$ and $X$, namely
\[
X^{-1}(I):=\{a\in A \mid \; \sca{X,aX}\subseteq I\},
\]
and
\[
J(I,X):=\{a\in A \mid [\phi_X(a)]_I\in \K([X]_I), aX^{-1}(I)\subseteq I\}.
\]
Note that $I$ does \emph{not} need to be positively invariant in order to make sense of these ideals.
Observe also that $A^{-1}(I)=I$ and $X^{-1}(I)\subseteq X^{-1}(J)$ whenever we have ideals $I,J\subseteq A$ satisfying $I\subseteq J$.
The use of the ideal $J(I,X)$ is pivotal in the work of Katsura \cite{Kat07} for accounting for $*$-representations of $\T_X$ that may not be injective on $X$.

As per \cite[Definitions 5.6 and 5.12]{Kat07}, we define a \emph{T-pair} of $X$ to be a pair $\L=\{\L_\mt,\L_{\{1\}}\}$ of ideals of $A$ such that $\L_\mt$ is positively invariant for $X$ and $\L_\mt \subseteq \L_{\{1\}} \subseteq J(\L_\mt,X)$; a T-pair $\L$ that satisfies $J_X \subseteq \L_{\{1\}}$ is called an \emph{O-pair}.
Theorem 8.6 and Proposition 8.8 of \cite{Kat07} are the key results that inspired the main theorems of \cite{BB24, DeK24}.

\begin{theorem}\label{T:Kat par} \cite[Theorem 8.6, Proposition 8.8]{Kat07}
Let $X$ be a C*-correspondence over a C*-algebra $A$.
Then there is a bijection between the set of T-pairs (resp. O-pairs) of $X$ and the set of gauge-invariant ideals of $\T_X$ (resp. $\O_X$) that preserves inclusions and intersections.
\end{theorem}

It should be noted that Theorem \ref{T:giut} is used in the proof of Theorem \ref{T:Kat par}.
The bijection of Theorem \ref{T:Kat par} restricts appropriately to a parametrisation of the gauge-invariant ideals of any relative Cuntz-Pimsner algebra \cite[Proposition 11.9]{Kat07}.

\subsection{Product systems}\label{Ss:prod sys}

Henceforth we will be working with the semigroup $\bZ_+^d$ (for $d\in\bN$) extensively.
Accordingly, we fix the following notation.
For $d\in\bN$, we write $[d] := \{1, 2, \dots, d\}$. 
We denote the usual free generators of $\bZ_+^d$ by $\un{1}, \dots, \un{d}$, and we set $\un{0} = (0, \dots, 0)$. 
For an element $\un{n}=(n_1,\dots,n_d)\in\bZ_+^d,$ we define the \emph{length} of $\un{n}$ by
\[ 
|\un{n}|:= \sum\{n_i\mid i\in[d]\}.
\]
For $\mt\neq F \subseteq [d]$, we write
\[ 
\un{1}_F := \sum\{\un{i}\mid i\in F\} \; \text{and} \; \un{1}_\mt:=\un{0}. 
\]
We consider the lattice structure on $\bZ_+^d$ given by
\[ 
\un{n} \vee \un{m} := (\max\{n_i, m_i\})_{i=1}^d \qand \un{n} \wedge \un{m} := (\min\{n_i, m_i\})_{i=1}^d. 
\]
The semigroup $\bZ_+^d$ imposes a partial order on $\bZ^d$ that is compatible with the lattice structure. 
Specifically, we say that $\un{n}\leq\un{m}$ (resp. $\un{n} < \un{m}$) if and only if $n_i\leq m_i$ for all $i\in[d]$ (resp. $\un{n} \leq \un{m}$ and $\un{n} \neq \un{m}$). 
We denote the \emph{support} of $\un{n}$ by 
\[
\supp \un{n} := \{i \in [d] \mid n_i \neq 0\},
\]
and we write
\[ 
\un{n} \perp \un{m} \iff \supp \un{n} \bigcap \supp \un{m} = \mt.
\]
For $F\subseteq[d]$, we write $\un{n} \perp F$ if $\supp \un{n} \bigcap F = \mt$.
Notice that the set $\{\un{n}\in\bZ_+^d\mid \un{n}\perp F\}$ is directed; indeed, if $\un{n},\un{m}\perp F$ then $\un{n},\un{m}\leq\un{n}\vee\un{m}\perp F$.
Consequently, we can make sense of limits with respect to $\un{n} \perp F$.

A \emph{product system $X$ over $\bZ_+^d$ with coefficients in a C*-algebra $A$} is a family $\{X_\un{n}\}_{\un{n}\in\bZ_+^d}$ of C*-correspondences over $A$ together with multiplication maps $u_{\un{n},\un{m}} \colon X_\un{n}\otimes_A X_\un{m}\to X_{\un{n}+\un{m}}$ for all $\un{n},\un{m}\in\bZ_+^d$, such that:
\begin{enumerate}
\item $X_\un{0}=A$, viewing $A$ as a C*-correspondence over itself in the usual way; 
\item if $\un{n}=\un{0}$, then $u_{\un{0},\un{m}} \colon A\otimes_A X_\un{m}\to[\phi_\un{m}(A)X_\un{m}]$ is the unitary implementing the left action of $A$ on $X_\un{m}$;
\item if $\un{m}=\un{0}$, then $u_{\un{n},\un{0}} \colon X_\un{n}\otimes_A A\to X_\un{n}$ is the unitary implementing the right action of $A$ on $X_\un{n}$;
\item if $\un{n},\un{m}\in \bZ_+^d\setminus\{\un{0}\}$, then $u_{\un{n},\un{m}} \colon X_\un{n}\otimes_A X_\un{m}\to X_{\un{n}+\un{m}}$ is a unitary;
\item the multiplication maps are associative in the sense that
\[ 
u_{\un{n}+\un{m},\un{k}}(u_{\un{n},\un{m}}\otimes\text{id}_{X_\un{k}})=u_{\un{n},\un{m}+\un{k}}(\text{id}_{X_\un{n}}\otimes u_{\un{m},\un{k}})\foral \un{n},\un{m},\un{k}\in \bZ_+^d. 
\]
\end{enumerate}

Note that we use $\phi_\un{n}$ to denote the left action $\phi_{X_\un{n}}$ of $X_\un{n}$ for each $\un{n}\in\bZ_+^d$.
We refer to the C*-correspondences $X_\un{n}$ as the \emph{fibres} of $X$.
We do \emph{not} assume that the fibres are non-degenerate.
Accordingly, a certain degree of care is required when working with the multiplication maps $u_{\un{0},\un{m}}$ for $\un{m}\in\bZ_+^d$.
If $X_\un{n}$ is injective/proper/regular for all $\un{n}\in\bZ_+^d$, then we say that $X$ is \emph{injective}/\emph{proper}/\emph{regular}.
For brevity, we will write 
\[
\xi_\un{n}\xi_\un{m}\equiv u_{\un{n},\un{m}}(\xi_\un{n}\otimes\xi_\un{m})\foral\xi_\un{n}\in X_\un{n}, \xi_\un{m}\in X_\un{m}\; \text{and} \; \un{n},\un{m}\in\bZ_+^d, 
\]
with the understanding that $\xi_\un{n}$ and $\xi_\un{m}$ are allowed to differ when $\un{n}=\un{m}$.
Axioms (i) and (ii) imply that the unitary $u_{\un{0},\un{0}} \colon A\otimes_A A\to A$ is simply multiplication in $A$.
Axioms (ii) and (v) imply that
\[
\phi_{\un{n}+\un{m}}(a)(\xi_\un{n}\xi_\un{m}) = (\phi_\un{n}(a)\xi_\un{n})\xi_\un{m} \foral \xi_\un{n}\in X_\un{n}, \xi_\un{m}\in X_\un{m}\; \text{and} \; \un{n},\un{m}\in\bZ_+^d.
\]
Note that the maps involved in axiom (v) are linear and bounded, and are therefore determined by their respective actions on simple tensors.

For $\un{n}\in\bZ_+^d\setminus\{\un{0}\}$ and $\un{m}\in\bZ_+^d$, we exploit the product system structure of $X$ to define a $*$-homomorphism $\iota_\un{n}^{\un{n}+\un{m}} \colon \L(X_\un{n})\to\L(X_{\un{n}+\un{m}})$ by 
\[
\iota_\un{n}^{\un{n}+\un{m}}(S)=u_{\un{n},\un{m}}(S\otimes\text{id}_{X_\un{m}})u_{\un{n},\un{m}}^* \foral S \in \L(X_\un{n}).
\]
In turn, we obtain that
\[
\textup{$\iota_\un{n}^{\un{n}+\un{m}}(S)(\xi_\un{n}\xi_\un{m})=(S\xi_\un{n})\xi_\un{m}$ for all $\xi_\un{n}\in X_\un{n}$ and $\xi_\un{m}\in X_\un{m}$. }
\]
We also define a $*$-homomorphism $\iota_\un{0}^\un{m}\colon\K(A)\to\L(X_\un{m})$ by $\iota_\un{0}^\un{m}( \phi_\un{0}(a))=\phi_\un{m}(a)$ for all $a\in A$.
Moreover, we have that 
\[
\iota_\un{n}^\un{n}=\text{id}_{\L(X_\un{n})} \foral \un{n}\in \bZ_+^d\setminus\{\un{0}\} \qand \iota_\un{0}^\un{0}=\id_{\K(A)}.
\]

The theory of product systems includes that of C*-correspondences in the sense that every C*-correspondence $X$ over a C*-algebra $A$ can be viewed as the product system $\{X_n\}_{n \in \bZ_+}$ with 
\[
X_0:=A \qand X_n := X^{\otimes n} \foral n\in\bN,
\] 
and multiplication maps $u_{n,m}$ for $n,m \neq 0$ given by the natural inclusions.

A \emph{(Toeplitz) representation} $(\pi,t)$ of $X$ on $\B(H)$ consists of a family $\{(\pi,t_\un{n})\}_{\un{n}\in\bZ_+^d}$, where $(\pi,t_\un{n})$ is a representation of $X_\un{n}$ on $\B(H)$ for all $\un{n}\in\bZ_+^d$, $t_\un{0}=\pi$ and
\[ 
t_\un{n}(\xi_\un{n})t_\un{m}(\xi_\un{m})=t_{\un{n}+\un{m}}(\xi_\un{n}\xi_\un{m}) \foral \xi_\un{n}\in X_\un{n}, \xi_\un{m}\in X_\un{m} \; \text{and} \; \un{n}, \un{m}\in\bZ_+^d.
\]
We write $\psi_\un{n}$ for the induced $*$-representation of $\K(X_\un{n})$ for all $\un{n}\in\bZ_+^d$. 
We say that $(\pi,t)$ is \emph{injective} if $\pi$ is injective; in this case $t_\un{n}$ and $\psi_\un{n}$ are isometric for all $\un{n}\in\bZ_+^d$.
We denote the C*-algebra generated by $\pi(A)$ and every $t_\un{n}(X_\un{n})$ by $\ca(\pi,t)$.
We write $\T_X$ for the universal C*-algebra with respect to the Toeplitz representations of $X$, and refer to it as the \emph{Toeplitz algebra (of $X$)}.
Note that $\T_X$ is generated by a universal Toeplitz representation $(\pi_X,t_X)$, and its universal property is captured as follows: if $(\pi,t)$ is a representation of $X$, then there exists a (unique) canonical $*$-epimorphism $\pi\times t\colon\T_X\to\ca(\pi,t)$.
Here canonicity means that $(\pi\times t)(t_{X,\un{n}}(\xi_\un{n}))=t_\un{n}(\xi_\un{n})$ for all $\xi_\un{n}\in X_\un{n}$ and $\un{n}\in\bZ_+^d$.


In practice, oftentimes $\T_X$ is too large to be useful. 
Instead, we use the structure of $\bZ_+^d$ to impose additional structure on $X$ and then study the representations that preserve it.
This leads to the consideration of C*-algebras that are more tractable to analyse than $\T_X$.
More precisely, we say that $X$ is \emph{compactly aligned} if for all $\un{n},\un{m}\in \bZ_+^d\setminus\{\un{0}\}$ we have that
\[ 
\iota_\un{n}^{\un{n}\vee\un{m}}(\K(X_\un{n})) \iota_\un{m}^{\un{n}\vee\un{m}}(\K(X_\un{m}))\subseteq \K(X_{\un{n}\vee\un{m}}).
\]
Notice that we disregard the case where $\un{n}$ or $\un{m}$ equals $\un{0}$, as the compact alignment condition holds automatically in this case.
Likewise, the compact alignment condition holds automatically when $d=1$.

Fixing a compactly aligned product system $X$ over $\bZ_+^d$ with coefficients in a C*-algebra $A$, a representation $(\pi,t)$ of $X$ is said to be \emph{Nica-covariant} if for all $\un{n},\un{m}\in\bZ_+^d\setminus\{\un{0}\}, k_\un{n}\in\K(X_\un{n})$ and $k_\un{m}\in\K(X_\un{m})$, we have that
\[
\psi_\un{n}(k_\un{n})\psi_\un{m}(k_\un{m})=\psi_{\un{n}\vee\un{m}}(\iota_\un{n}^{\un{n}\vee\un{m}}(k_\un{n})\iota_\un{m}^{\un{n}\vee\un{m}}(k_\un{m})).
\]
We disregard the case where $\un{n}$ or $\un{m}$ equals $\un{0}$, as the Nica-covariance condition holds automatically in this case.
The Nica-covariance condition induces a Wick ordering on $\ca(\pi,t)$, e.g., \cite{DKKLL20, Fow02, KL19a, KL19b}.
More precisely, for $\un{n},\un{m}\in\bZ_+^d$, we have that
\[
t_\un{n}(X_\un{n})^*t_\un{m}(X_\un{m})\subseteq[t_{\un{n}'}(X_{\un{n}'})t_{\un{m}'}(X_{\un{m}'})^*], \; \text{where} \; \un{n}'=\un{n}\vee\un{m}-\un{n} \; \text{and} \; \un{m}'=\un{n}\vee\un{m}-\un{m},
\]
from which it follows that
\[
\ca(\pi,t)=\ol{\spn}\{t_\un{n}(X_\un{n})t_\un{m}(X_\un{m})^*\mid\un{n},\un{m}\in\bZ_+^d\}.
\]

We write $\N\T_X$ for the universal C*-algebra with respect to the Nica-covariant representations of $X$, and refer to it as the \emph{Toeplitz-Nica-Pimsner algebra (of $X$)}.
Since the Nica-covariance relations are graded, the existence of $\N\T_X$ and its universal property follow from the corresponding properties of $\T_X$.
We write $(\ol{\pi}_X, \ol{t}_X)$ for the \emph{universal Nica-covariant representation (of $X$)}.
If $(\pi,t)$ is a Nica-covariant representation of $X$, we will write (in a slight abuse of notation) $\pi \times t$ for the canonical $*$-epimorphism $\N\T_X \to \ca(\pi,t)$.
Since $\bZ_+^d$ is contained in the amenable discrete group $\bZ^d$, the C*-algebra $\N\T_X$ can also be realised concretely via a Fock space construction.
This property was exploited frequently in \cite{DeK24}, though we will not need it here.

We say that a Nica-covariant representation $(\pi,t)$ of $X$ \emph{admits a gauge action} $\ga$ if there exists a family $\{\ga_\un{z}\}_{\un{z}\in\bT^d}$ of $*$-endomorphisms of $\ca(\pi,t)$ satisfying
\[ 
\ga_\un{z}(\pi(a))=\pi(a) \foral a\in A \; \text{and} \; \ga_\un{z}(t_\un{n}(\xi_\un{n}))=\un{z}^\un{n}t_\un{n}(\xi_\un{n}) \foral \xi_\un{n}\in X_\un{n} \; \text{and} \; \un{n}\in\bZ_+^d\setminus\{\un{0}\}, 
\]
for each $\un{z}\in\bT^d$.
If $\un{z}=(z_1,\dots,z_d)\in\bT^d$ and $\un{n}=(n_1,\dots,n_d)\in\bZ_+^d$, then $\un{z}^\un{n}:=\prod_{j=1}^dz_j^{n_j}$.
When such a gauge action $\ga$ exists, it is necessarily unique. 
We also have that each $\ga_\un{z}$ is a $*$-automorphism, the family $\{\ga_\un{z}\}_{\un{z}\in\bT^d}$ is point-norm continuous, and we obtain a group homomorphism (also denoted by $\ga$) defined by
\[
\ga\colon\bT^d\to\text{Aut}(\ca(\pi,t)); \un{z}\mapsto\ga_\un{z}\foral \un{z}\in\bT^d.
\]
The universal Nica-covariant representation of $X$ admits a gauge action.
We say that an ideal $\mathfrak{J}\subseteq\ca(\pi,t)$ is \emph{gauge-invariant} or \emph{equivariant} if $\ga_\un{z}(\mathfrak{J})\subseteq\mathfrak{J}$ for all $\un{z}\in\bT^d$ (and so $\ga_\un{z}(\mathfrak{J})=\mathfrak{J}$ for all $\un{z}\in\bT^d$).

Given $\un{m},\un{m}'\in\bZ_+^d$ with $\un{m} \leq \un{m}'$, we write
\begin{align*}
B_{[\un{m}, \un{m}']}^{(\pi,t)} := \spn\{ \psi_{\un{n}}(\K(X_{\un{n}})) \mid \un{m} \leq \un{n} \leq \un{m}'\}
\qand
B_{(\un{m}, \un{m}']}^{(\pi,t)} := \spn\{ \psi_{\un{n}}(\K(X_{\un{n}})) \mid \un{m} < \un{n} \leq \un{m}'\}.
\end{align*}
These spaces are in fact C*-subalgebras of $\ca(\pi,t)$, e.g, \cite{CLSV11}.
By convention we take the linear span of $\mt$ to be $\{0\}$, so that $B^{(\pi,t)}_{(\un{m},\un{m}]}=\{0\}$ for all $\un{m}\in\bZ_+^d$.
We also define
\begin{align*}
B_{[\un{m}, \infty]}^{(\pi,t)} := \ol{\spn} \{ \psi_{\un{n}}(\K(X_{\un{n}})) \mid \un{m} \leq \un{n} \}
\qand
B_{(\un{m}, \infty]}^{(\pi,t)} := \ol{\spn} \{ \psi_{\un{n}}(\K(X_{\un{n}})) \mid \un{m} < \un{n} \}.
\end{align*}
We refer to these C*-algebras as the \emph{cores} of $(\pi,t)$. 
When $(\pi,t)$ admits a gauge action $\ga$, we have that
\[ 
B_{[\un{0},\infty]}^{(\pi,t)}=\ca(\pi,t)^\ga:=\{f\in\ca(\pi,t) \mid \ga_{\un{z}}(f)=f \foral \un{z}\in\bT^d\}, 
\]
where $\ca(\pi,t)^\ga$ is the \emph{fixed point algebra} of $\ca(\pi,t)$ (with respect to $\gamma$).

Describing the Cuntz-type object of $X$ is more challenging than in the case of a single C*-correspondence; see \cite{Seh18, SY10} for further details.
To alleviate this difficulty, we will make a further structural imposition on $X$, introduced by Dor-On and Kakariadis \cite{DK18}.
Let $X$ be a product system over $\bZ_+^d$ with coefficients in a C*-algebra $A$.
We say that $X$ is \emph{strong compactly aligned} if it is compactly aligned and satisfies
\begin{equation}\label{Eq:sca}
\iota_\un{n}^{\un{n}+\un{i}}(\K(X_\un{n}))\subseteq\K(X_{\un{n}+\un{i}}) 
\textup{ whenever $\un{n}\perp\un{i}$, where $i\in[d]$ and $\un{n}\in\bZ_+^d\setminus\{\un{0}\}$}.
\end{equation}
We disallow $\un{n}=\un{0}$, as then (\ref{Eq:sca}) would imply that the strong compactly aligned product systems are exactly the proper product systems over $\bZ_+^d$ (see \cite[Proposition 2.5.1]{DeK24} and Proposition \ref{P:propimpsca} to come). 
Note that (\ref{Eq:sca}) does not imply compact alignment (rather, a strong compactly aligned product system is \emph{a priori} assumed to be compactly aligned). 
Any C*-correspondence, when viewed as a product system over $\bZ_+$, is vacuously strong compactly aligned.
Not every strong compactly aligned product system is proper \cite[Example 7.4]{DK18}; however, every proper product system over $\bZ_+^d$ is strong compactly aligned.
More precisely, we have the following proposition.

\begin{proposition}\label{P:propimpsca}
Let $X$ be a proper product system over $\bZ_+^d$ with coefficients in a C*-algebra $A$.
Then $\iota_\un{n}^{\un{n}+\un{m}}(\K(X_\un{n}))\subseteq\K(X_{\un{n}+\un{m}})$ for all $\un{n},\un{m}\in\bZ_+^d$, and thus $X$ is strong compactly aligned.
\end{proposition}
\begin{proof}
The result follows immediately by \cite[Proposition 4.7]{Lan95}.
\end{proof}

We will require some notation and results from \cite{DK18}.
Henceforth, we assume that $X$ is strong compactly aligned.
Firstly, strong compact alignment yields that
\[
\bigcap_{i\in F}\phi_{\un{i}}^{-1}(\K(X_{\un{i}}))=\bigcap\{\phi_{\un{n}}^{-1}(\K(X_{\un{n}}))\mid \un{0}\leq\un{n}\leq\un{1}_F \}\foral \mt\neq F\subseteq[d].
\]
For each $\mt\neq F\subseteq[d]$, we define
\[ 
\J_F :=
(\bigcap_{i\in F}\ker\phi_{\un{i}})^\perp\cap(\bigcap_{i\in [d]}\phi_{\un{i}}^{-1}(\K(X_{\un{i}})))\qand \J_\mt:=\{0\},
\]
which are ideals of $A$.
In turn, for each $F\subseteq [d]$, we define
\[ 
\I_F :=
\{a \in A \mid \sca{X_{\un{n}},aX_{\un{n}}}\subseteq \J_F \; \text{for all} \; \un{n}\perp F\}=\bigcap\{X_\un{n}^{-1}(\J_F)\mid \un{n}\perp F\}.
\]
In particular, we have that $\I_\mt=\{0\}$ and $\I_F \subseteq \J_F$ for all $F \subseteq [d]$.
The ideal $\I_F$ is the largest ideal in $\J_F$ that is $F^\perp$-invariant \cite[Proposition 2.7]{DK18}.
To avoid ambiguity, given two strong compactly aligned product systems $X$ and $Y$, we will denote the ideals $\J_F$ (resp. $\I_F$) for $X$ and $Y$ by $\J_F(X)$ and $\J_F(Y)$ (resp. $\I_F(X)$ and $\I_F(Y)$), respectively.


The ideals $\I_F$ emerge naturally when solving polynomial equations, originating in \cite{DFK17} in the case of C*-dynamical systems.
In order to make this precise, we require the following notation.
Following the conventions of \cite[Section 3]{DK18}, we introduce an approximate unit $(k_{\un{i},\la})_{\la\in\La}$ of $\K(X_{\un{i}})$ for each generator $\un{i}$ of $\bZ_+^d$.
Without loss of generality, we may assume that these approximate units are indexed by the same directed set $\La$, by replacing with their product.
Let $(\pi,t)$ be a Nica-covariant representation of $X$. 
Fixing $i \in [d]$, we define
\begin{equation*}
p_{\un{i},\la}:=\psi_{\un{i}}(k_{\un{i},\la}) \foral \la\in\La,\; \text{and} \; p_{\un{i}}:=\textup{w*-}\lim_\la p_{\un{i},\la},
\end{equation*}
i.e., $p_{\un{i}}$ is the projection on the space $[\psi_{\un{i}}(\K(X_{\un{i}}) ) H]$ for the Hilbert space $H$ on which $(\pi,t)$ acts.
In turn, we set
\begin{equation*}
q_\mt:= I_H
\text{ and }
q_F:=\prod_{i\in F}(I_H-p_{\un{i}}) \textup{ for all $\mt\neq F\subseteq [d]$}.
\end{equation*}
It should be noted that the projections $p_\un{i}$ commute \cite[Remark 2.5.8]{DeK24}, so there is no ambiguity regarding the order of the product defining each $q_F$.
Additionally, one can make sense of the projections $p_\un{i}$ even if $X$ is (just) compactly aligned.
We gather some algebraic relations proved in \cite{DeK24}.
Alternate proofs are provided in \cite{De24} which capitalise on the aforementioned commutativity of the projections $p_\un{i}$.

\begin{proposition}\label{P:sca ai} \cite[Proposition 2.4]{DK18}
Let $X$ be a strong compactly aligned product system with coefficients in a C*-algebra $A$.
Let $(k_{\un{i},\la})_{\la\in\La}$ be an approximate unit of $\K(X_\un{i})$ for all $i\in[d]$.
Fix $\mt\neq F \subseteq [d]$ and $\un{0} \neq \un{n} \in \bZ_+^d$, and set $\un{m} = \un{n} \vee \un{1}_F$.
Then the net $(e_{F,\la})_{\la\in\La}$ defined by
\[
e_{F, \la}: = \prod \{ \iota_{\un{i}}^{\un{1}_F}(k_{\un{i}, \la}) \mid i \in F \} \foral \la\in\La
\]
is contained in $\K(X_{\un{1}_F})$, and we have that
\begin{equation}\label{eq2-1}
\nor{\cdot}\text{-}\lim_\la \iota_{\un{1}_F}^{\un{m}}(e_{F, \la})\iota_{\un{n}}^{\un{m}}(k_{\un{n}}) 
= 
\iota_{\un{n}}^{\un{m}}(k_{\un{n}})  \foral k_\un{n} \in \K(X_{\un{n}}).
\end{equation}
Moreover, it follows that $\iota_{\un{n}}^{\un{m}}(k_{\un{n}}) \in \K(X_{\un{m}})$ for all $k_\un{n} \in \K(X_{\un{n}})$.
\end{proposition}

It should be noted that (\ref{eq2-1}) holds independently of the order of the product defining $e_{F,\la}$.

\begin{proposition}\label{P:pf reducing}\cite[Proposition 4.4]{DK18}
Let $X$ be a strong compactly aligned product system with coefficients in a C*-algebra $A$.
Let $(\pi,t)$ be a Nica-covariant representation of $X$ and fix $F \subseteq [d]$.
Then for all $\un{m}\in\bZ_+^d$ and $\xi_{\un{m}} \in X_{\un{m}}$, we have that
\[
 q_Ft_{\un{m}}(\xi_{\un{m}})
=
\begin{cases}
t_{\un{m}}(\xi_{\un{m}})q_F & \text{ if } \un{m} \perp F, \\
0 & \text{ if } \un{m} \not\perp F,
\end{cases}
\]
so that in particular $q_F \in \pi(A)'$.
\end{proposition}

\begin{proposition}\label{P:prod cai} \cite[Section 3]{DK18}
Let $X$ be a strong compactly aligned product system with coefficients in a C*-algebra $A$ and let $(\pi,t)$ be a Nica-covariant representation of $X$.
Fixing $\mt \neq F \subseteq [d]$, we have that
\[
\nor{\cdot}\text{-}\lim_\la  \psi_{\un{n}}(k_{\un{n}}) \prod_{i \in F} p_{\un{i}, \la}
=
\psi_{\un{n}}(k_{\un{n}}) \prod_{i \in F} p_{\un{i}}
\foral \un{n}\in\bZ_+^d\setminus\{\un{0}\} \; \text{and} \; k_{\un{n}} \in \K(X_{\un{n}}).
\]
If $a \in \bigcap\{ \phi_{\un{i}}^{-1}( \K(X_{\un{i}})) \mid i \in F \}$, then 
\[
\pi(a) \prod_{i \in D} p_{\un{i}} = \nor{\cdot}\text{-}\lim_\la\pi(a)\prod_{i\in D}p_{\un{i},\la}=\psi_{\un{1}_D}(\phi_{\un{1}_D}(a)) \foral \mt\neq D \subseteq F,
\]
and so
\[
\pi(a)q_F
=
\pi(a) + \sum \{ (-1)^{|\un{n}|} \psi_{\un{n}}(\phi_{\un{n}}(a)) \mid \un{0} \neq \un{n} \leq \un{1}_F \}
\in \ca(\pi,t).
\]
\end{proposition}

\begin{proposition}\label{P:DK3.3}\cite[Proposition 3.3]{DK18}
Let $X$ be a strong compactly aligned product system with coefficients in a C*-algebra $A$.
Suppose that $(\pi,t)$ is a Nica-covariant representation of $X$ and fix $a\in A$.
If there exist $\un{m}\in\bZ_+^d\setminus\{\un{0}\}$ and $k_\un{n}\in\K(X_\un{n})$ for each $\un{0}\neq\un{n}\leq\un{m}$ such that
\[
\pi(a)+\sum\{\psi_\un{n}(k_\un{n})\mid\un{0}\neq\un{n}\leq\un{m}\}=0,
\]
then we have that
\begin{equation*}
\pi(a)q_F=0 \; \text{for} \; F:=\supp\un{m}.
\end{equation*}
\end{proposition}

The following proposition justifies the usage of the family $\I$.

\begin{proposition}\label{P:DK3.4}\cite[Proposition 3.4]{DK18}
Let $X$ be a strong compactly aligned product system with coefficients in a C*-algebra $A$.
Suppose that $(\pi,t)$ is an injective Nica-covariant representation of $X$ and fix $a\in A$ and $\un{m}\in\bZ_+^d$.
If $\pi(a)\in B_{(\un{0},\un{m}]}^{(\pi,t)}$, then $a\in\I_F$ for $F:=\supp\un{m}$.
\end{proposition}

We define the \emph{ideal of the CNP-relations} by
\begin{equation}\label{Eq:CNPideal}
\mathfrak{J}_\I:=\sca{\ol{\pi}_X(\I_F)\ol{q}_{X,F}\mid F\subseteq[d]}\subseteq\N\T_X. 
\end{equation}
We then define the \emph{Cuntz-Nica-Pimsner algebra (of $X$)} to be the following C*-algebra:
\[
\N\O_X :=\N\T_X/\mathfrak{J}_\I.
\] 
This C*-algebra is universal with respect to the \emph{CNP-representations of $X$}, i.e., those Nica-covariant representations $(\pi,t)$ of $X$ that also satisfy
\[
\pi(a)q_F = \pi(a) + \sum \{ (-1)^{|\un{n}|} \psi_{\un{n}}(\phi_{\un{n}}(a)) \mid \un{0} \neq \un{n} \leq \un{1}_F \}=0  \foral a\in\I_F \; \text{and} \; F\subseteq[d].
\]
We can view $\N\O_X$ as the C*-algebra generated by a universal CNP-representation of $X$, and this representation admits a gauge action since $\N\O_X$ is an equivariant quotient of $\N\T_X$.
Notice that $\N\O_X$ is defined with respect to simple algebraic relations (by Proposition \ref{P:prod cai}) that are induced by $2^d$ ideals of the coefficient algebra, namely the family $\I$.
This construction resembles that of the Cuntz-Pimsner algebra of a single C*-correspondence, and recovers it when $d=1$.

In \cite{DK18} it is shown that $\N\O_X$ coincides with the Cuntz-Nica-Pimsner algebra of Sims and Yeend \cite{SY10}, and thus with the strong covariance algebra of Sehnem \cite{Seh18}.
In particular, the universal CNP-representation is injective by \cite[Theorem 4.1]{SY10}, since $(\bZ^d,\bZ_+^d)$ satisfies \cite[(3.5)]{SY10}.
Moreover, $\N\O_X$ is co-universal with respect to the injective Nica-covariant representations of $X$ that admit a gauge action \cite{SY10}.
The co-universal property of $\N\O_X$ has been verified in several works \cite{CLSV11, DK18, DKKLL20, Seh21} in more general contexts.

We close this section by outlining how the quotient C*-correspondence construction can be extended to product systems.
Let $X$ be a product system over $\bZ_+^d$ with coefficients in a C*-algebra $A$ and let $I\subseteq A$ be an ideal.
We say that $I$ is \emph{positively invariant (for $X$)} if it satisfies
\[
\ol{\spn}\{\sca{X_\un{n},IX_\un{n}}\mid \un{n}\in\bZ_+^d\}\subseteq I.
\]
In other words, the ideal $I$ is positively invariant for $X$ if and only if it is positively invariant for every fibre of $X$.
This observation lies at the heart of the following proposition.

\begin{proposition}\label{P:qntrprodsys}\cite[Propositions 2.3.5, 2.4.4 and 2.5.5]{DeK24}
Let $X$ be a product system over $\bZ_+^d$ with coefficients in a C*-algebra $A$ and let $I\subseteq A$ be an ideal that is positively invariant for $X$.
Set 
\[
[X]_I:=\{[X_\un{n}]_I\}_{\un{n}\in\bZ_+^d},
\; \text{where } \;
[X_\un{n}]_I = X_\un{n}/X_\un{n}I \foral \un{n} \in \bZ_+^d.
\]
Then $[X]_I$ carries a canonical structure as a product system over $\bZ_+^d$ with coefficients in $[A]_I$, given by the multiplication maps
\[
[X_\un{n}]_I\otimes_{[A]_I}[X_\un{m}]_I\to[X_{\un{n}+\un{m}}]_I; [\xi_\un{n}]_I \otimes [\xi_\un{m}]_I \mapsto [\xi_\un{n} \xi_\un{m}]_I \foral \xi_\un{n}\in X_\un{n}, \xi_\un{m}\in X_\un{m}, \un{n},\un{m} \in \bZ_+^d.
\] 
Additionally, if $X$ is (strong) compactly aligned, then so is $[X]_I$.
\end{proposition}

\section{Gauge-invariant ideal structure of $\N\T_X$}\label{S:BBDeK}

\noindent Next we present the more specialised tools used to obtain the main results of \cite{BB24, DeK24}.
Most of the material in this section constitutes an abridged account of \cite{BB24,DeK24}, though there are some new results (Propositions \ref{P:maxNT}, \ref{P:repTfam} and \ref{P:Tfamga}, and Lemma \ref{L:GIUTequiv}) to show how the objects of interest fit within the broader landscape of the aforementioned works.

\subsection{NT-$2^d$-tuples}\label{Ss:NT}

We begin by summarising the tools and concepts used to arrive at the main result of \cite{DeK24}.
Therein the analysis proceeds by first dealing with the ``injective" case and then using the quotient product system machinery to deal with the ``non-injective" case.
The meaning behind this nomenclature will be clarified in the sequel.
Henceforth, we will take $X$ to be a strong compactly aligned product system with coefficients in a C*-algebra $A$.
Drawing inspiration from Theorem \ref{T:Kat par} and the role of Theorem \ref{T:giut} within the proof, first we need to extend the relative Cuntz-Pimsner algebra construction.
This leads to the following definition.

\begin{definition}\label{D:CNP rel}\cite[Definition 3.1.1]{DeK24}
Let $X$ be a strong compactly aligned product system with coefficients in a C*-algebra $A$. 
A \emph{$2^d$-tuple (of $X$)} is a family $\L := \{\L_F\}_{F \subseteq [d]}$ such that $\L_F$ is a non-empty subset of $A$ for all $F\subseteq[d]$.
A $2^d$-tuple $\L$ of $X$ is called \emph{relative} if
\[
\L_F\subseteq \bigcap\{\phi_{\un{i}}^{-1}(\K(X_{\un{i}}))\mid i\in F\} 
\foral
\mt\neq F\subseteq [d].
\]
\end{definition}

The consideration of families of $2^d$ non-empty subsets of the coefficient algebra is inspired by the family $\I$.
We write $\L\subseteq\L'$ for $2^d$-tuples $\L$ and $\L'$ if and only if $\L_F\subseteq\L_F'$ for all $F\subseteq[d]$.
This defines a partial order on the set of $2^d$-tuples of $X$.
We say that $\L=\L'$ if and only if $\L\subseteq\L'$ and $\L'\subseteq\L$.
Two key (relative) $2^d$-tuples are $\{\{0\}\}_{F\subseteq[d]}$ and $\I$.

Let $(\pi,t)$ be a Nica-covariant representation of $X$.
The crucial property of a relative $2^d$-tuple $\L$ is that
\[
\pi(a)q_F=\pi(a)+\sum\{(-1)^{|\un{n}|}\psi_\un{n}(\phi_\un{n}(a))\mid \un{0}\neq\un{n}\leq\un{1}_F\}\in\ca(\pi,t)\foral a\in\L_F \; \text{and} \; \mt\neq F\subseteq[d],
\]
using Proposition \ref{P:prod cai}.
This allows us to extend the ideas of \cite{DK18} in a natural way.

\begin{definition}\label{D:LCNPideal}\cite[Definition 3.1.3]{DeK24}
Let $X$ be a strong compactly aligned product system with coefficients in a C*-algebra $A$ and let $\L$ be a relative $2^d$-tuple of $X$. 
We define the \emph{ideal of the $\L$-relative CNP-relations} to be
\[
\fJ_{\L} := \sum\{\fJ_{\L, F}\mid F\subseteq[d]\} \subseteq \N\T_X,
\; \text{where} \;
\mathfrak{J}_{\L,F}:=\sca{\ol{\pi}_X(\L_F)\ol{q}_{X,F}} \; \text{for each} \; F\subseteq[d].
\]
We say that \emph{$\L$ induces $\fJ_{\L}$}.
\end{definition}

Being an algebraic sum of ideals in $\N\T_X$, the space $\mathfrak{J}_\L$ is itself an ideal in $\N\T_X$.
In turn, we obtain that
\[
\fJ_\L=\sca{\ol{\pi}_X(\L_F)\ol{q}_{X,F}\mid F\subseteq[d]}.
\]
By setting $\L=\I$, we recover $\mathfrak{J}_\I$ as defined in \cite{DK18} and (\ref{Eq:CNPideal}).
It is routine to check that $\fJ_\L$ and each $\fJ_{\L,F}$ are gauge-invariant (or see \cite[Proposition 3.1.4]{De24}).

Usually we are more interested in the ideal $\fJ_\L$ than the relative $2^d$-tuple $\L$.
Thus, noting that the family $\sca{\L}:=\{\sca{\L_F}\}_{F\subseteq[d]}$ is a relative $2^d$-tuple and that $\fJ_\L=\fJ_{\sca{\L}}$ \cite[Lemma 3.1.4]{DeK24}, in many cases we can assume that $\L$ consists of ideals without loss of generality.

Relative $2^d$-tuples are so-named because they give rise to the appropriate higher-rank analogue of relative Cuntz-Pimsner algebras.

\begin{definition}\label{D:LCNP}\cite[Definition 3.1.15]{DeK24}
Let $X$ be a strong compactly aligned product system with coefficients in a C*-algebra $A$. 
Let $\L$ be a relative $2^d$-tuple of $X$ and let $(\pi,t)$ be a Nica-covariant representation of $X$. 
We say that $(\pi,t)$ is an \emph{$\L$-relative CNP-representation (of $X$)} if it satisfies
\[
\pi(\L_F)q_F = \{0\}
\foral
F \subseteq [d].
\]
The universal C*-algebra with respect to the $\L$-relative CNP-representations of $X$ is denoted by $\N\O(\L,X)\equiv\N\T_X/\fJ_\L$, and we refer to it as the \emph{$\L$-relative Cuntz-Nica-Pimsner algebra (of $X$)}.
We write $(\pi_X^\L,t_X^\L)$ for the \emph{universal $\L$-relative CNP-representation (of $X$)}.
\end{definition}

Existence and uniqueness (up to canonical $*$-isomorphism) of the pair $(\N\O(\L,X),(\pi_X^\L,t_X^\L))$ are ascertained in \cite[Proposition 3.1.16]{DeK24}.
Since $\N\O(\L,X)$ is an equivariant quotient of $\N\T_X$, the representation $(\pi_X^\L,t_X^\L)$ admits a gauge action.
Notice that
\[
\N\O(\{\{0\}\}_{F\subseteq[d]},X)=\N\T_X\qand \N\O(\I,X)=\N\O_X.
\]
When $X$ is a C*-correspondence, this construction recovers the relative Cuntz-Pimsner algebras.
When working with $\N\O(\L,X)$, we can assume that $\L$ consists of ideals without loss of generality, since $\N\O(\L,X)=\N\O(\sca{\L},X)$ by the comments preceding Definition \ref{D:LCNP}.
A key question is to ascertain the conditions which need to be imposed on $\L$ in order for $\N\O(\L,X)$ to satisfy a Gauge-Invariant Uniqueness Theorem.
As we will see, this question is intrinsically linked to the question of parametrising the gauge-invariant ideals of $\N\T_X$.

A first approach towards the parametrisation of the gauge-invariant ideals of $\N\T_X$ would be to establish a correspondence between the relative $2^d$-tuples of $X$ and the gauge-invariant ideals of $\N\T_X$ that they induce.
However, this is insufficient as different relative $2^d$-tuples may induce the same gauge-invariant ideal of $\N\T_X$ \cite[Remark 3.1.6]{DeK24}.
To remedy this issue, we instead look for the largest relative $2^d$-tuple inducing a fixed gauge-invariant ideal of $\N\T_X$.

\begin{definition}\label{D:maximal}\cite[Definition 3.1.8]{DeK24}
Let $X$ be a strong compactly aligned product system with coefficients in a C*-algebra $A$ and let $\M$ be a relative $2^d$-tuple of $X$.
We say that $\M$ is a \emph{maximal $2^d$-tuple (of $X$)} if whenever $\L$ is a relative $2^d$-tuple of $X$ such that $\mathfrak{J}_\L=\mathfrak{J}_\M$ and $\M\subseteq\L$, we have that $\M=\L$.
\end{definition}

Given a relative $2^d$-tuple $\L$, there always exists a maximal $2^d$-tuple $\M$ which induces $\fJ_\L$.
This maximal $2^d$-tuple is unique and consists of ideals \cite[Proposition 3.1.9]{DeK24}.
Note that both $\{\{0\}\}_{F\subseteq[d]}$ and $\I$ are maximal \cite[Remarks 3.1.10 and 3.2.8]{DeK24}.

Injective Nica-covariant representations that admit gauge actions provide the quintessential supply of maximal $2^d$-tuples.
More precisely, let $(\pi,t)$ be a Nica-covariant representation. 
We define $\L^{(\pi,t)}$ to be the $2^d$-tuple of $X$ given by
\[
\L_\mt^{(\pi,t)}:=\ker\pi \qand \L_F^{(\pi,t)}:=\pi^{-1}(B_{(\un{0},\un{1}_F]}^{(\pi,t)}) \foral \mt\neq F\subseteq[d].
\]
It is straightforward to check that $\L^{(\pi,t)}$ consists of ideals.
When $(\pi,t)$ is injective and admits a gauge action, the $2^d$-tuple $\L^{(\pi,t)}$ is maximal and contained in $\I$ \cite[Proposition 3.1.18]{DeK24}.

The property of being contained inside $\I$ is a useful one that a $2^d$-tuple $\L$ may or may not possess (e.g., any $2^d$-tuple $\L$ with $\L_\mt\neq\{0\}$ is not contained inside $\I$).
The study of $2^d$-tuples $\L$ satisfying $\L\subseteq\I$ was central to \cite{DeK24}, and the key advantage is that they are exactly the relative $2^d$-tuples $\L$ such that $\N\O(\L,X)$ contains an isometric copy of $X$ \cite[Proposition 3.2.1]{DeK24}.
In turn, the structure of $\N\O(\L,X)$ permits an analysis via cores \cite[Proposition 3.1.17]{DeK24}.
It follows that the $2^d$-tuples $\L$ which are both inside $\I$ and maximal admit the following Gauge-Invariant Uniqueness Theorem.

\begin{theorem}\label{T:d GIUT M}\cite[Theorem 3.2.11]{DeK24}
Let $X$ be a strong compactly aligned product system with coefficients in a C*-algebra $A$. 
Let $\L$ be a maximal $2^d$-tuple of $X$ such that $\L\subseteq\I$ and let $(\pi,t)$ be a Nica-covariant representation of $X$. 
Then $\N\O(\L,X)\cong\ca(\pi,t)$ via a (unique) canonical $*$-isomorphism if and only if $(\pi,t)$ admits a gauge action and $\L^{(\pi,t)}=\L$.
\end{theorem}

Theorem \ref{T:d GIUT M} recovers \cite[Corollary 11.8]{Kat07} when $d=1$, and \cite[Theorem 4.2]{DK18} when $\L = \I$.
The $2^d$-tuples $\L$ that satisfy $\L\subseteq\I$ but which may not be maximal also admit a Gauge-Invariant Uniqueness Theorem \cite[Theorem 3.4.9]{DeK24}, though we will not need it in the current work. 

It follows that the maximal $2^d$-tuples $\L$ of $X$ that are inside $\I$ parametrise the gauge-invariant ideals $\fJ$ of $\N\T_X$ such that the Nica-covariant representation $(Q_\fJ\circ\ol{\pi}_X,Q_\fJ\circ\ol{t}_X)$ is injective, where $Q_\fJ\colon\N\T_X\to\N\T_X/\fJ$ is the quotient map \cite[Remark 3.2.10]{DeK24}.
This justifies the naming convention used at the start of the section.
Note that $Q_\fJ\circ\ol{t}_X$ is notational shorthand for the family of maps $\{Q_\fJ\circ\ol{t}_{X,\un{n}}\}_{\un{n}\in\bZ_+^d}$.
The representation $(Q_\fJ\circ\ol{\pi}_X,Q_\fJ\circ\ol{t}_X)$ admits a gauge action by gauge-invariance of $\fJ$.

Before moving on to the ``non-injective" case, we turn to characterising maximality of a $2^d$-tuple $\L\subseteq\I$ via product system operations alone.
There are four ingredients in this respect, and we have already seen one: $\L$ needs to consist of ideals.
The next two are easily extracted and abstracted from $\I$. 

\begin{definition}\label{D:inv po}\cite[Definition 3.1.11]{DeK24}
Let $X$ be a strong compactly aligned product system with coefficients in a C*-algebra $A$. 
Let $\L$ be a $2^d$-tuple of $X$.
\begin{enumerate}
\item  We say that $\L$ is $X$-\emph{invariant} if $\left[\sca{X_\un{n},\L_F X_\un{n}}\right]\subseteq\L_F$ for all $\un{n}\perp F$ and $F\subseteq[d]$.
\item We say that $\L$ is \emph{partially ordered} if $\L_{F_1} \subseteq \L_{F_2}$ whenever  $F_1 \subseteq F_2 \subseteq [d]$.
\end{enumerate}
\end{definition}
When the underlying product system $X$ is clear from the context, we will abbreviate ``$X$-invariant" as simply ``invariant".
Notice that when we take $F=\mt$, condition (i) implies that $\L_\mt$ is positively invariant for $X$ (provided that $\L_\mt$ is an ideal).
If $\L_F$ is an ideal, then we may drop the closed linear span in condition (i).
The $2^d$-tuple $\I$ is invariant and partially ordered, and so is $\L^{(\pi,t)}$ for any Nica-covariant representation $(\pi,t)$ \cite[Proposition 3.1.14]{DeK24}.
The $2^d$-tuple $\J$ is partially ordered but is not, in general, invariant (see Remark \ref{R:Jnotinv} for a counterexample).

The final ingredient of maximality necessitates some auxiliary objects.

\begin{definition}\label{D:L1}\cite[Definition 3.4.1]{DeK24}
Let $X$ be a strong compactly aligned product system with coefficients in a C*-algebra $A$ and let $\L$ be a $2^d$-tuple of $X$ that consists of ideals.
Fixing $\mt \neq F \subsetneq [d]$, we define
\[
\L_{\inv, F} := \bigcap_{\un{n}\perp F}X_\un{n}^{-1}(\cap_{F\subsetneq D}\L_D)
\qand
\L_{\lim, F} := \{a\in A \mid \lim_{\un{n}\perp F}\|\phi_\un{n}(a)+\K(X_\un{n}\L_F)\|=0\}.
\]
If in addition $\L\subseteq\I$, then we define the $2^d$-tuple $\L^{(1)}$ of $X$ by
\[
\L_F^{(1)} 
:=
\begin{cases}
\{0\} & \text{ if } F = \mt, \\
\I_F \cap \L_{\inv, F} \cap \L_{\lim, F} & \text{ if } \mt\neq F \subsetneq [d], \\
\L_{[d]} & \text{ if } F = [d].
\end{cases}
\]
\end{definition}

Both $\L_{\inv,F}$ and $\L_{\lim,F}$ are ideals of $A$ \cite[Proposition 3.4.2]{DeK24}.
The equality 
\[
\lim_{\un{n}\perp F}\|\phi_\un{n}(a)+\K(X_\un{n}\L_F)\|=0 
\]
holds if and only if for each $\varepsilon>0$ there exists $\un{n}\perp F$ such that
\[
\|\phi_\un{m}(a)+\K(X_\un{m}\L_F)\|<\varepsilon\foral \un{m}\geq\un{n} \; \text{satisfying} \; \un{m}\perp F.
\]
This condition simplifies when we impose additional structure on $\L$.
More precisely, we have the following proposition.

\begin{proposition}\label{P:lim perp}\cite[Lemma 3.3.3]{DeK24}
Let $X$ be a strong compactly aligned product system with coefficients in a C*-algebra $A$.
Let $\L$ be an invariant $2^d$-tuple of $X$ which consists of ideals and satisfies the following condition:
\[
\L_F\subseteq\bigcap\{\phi_\un{i}^{-1}(\K(X_\un{i}))\mid i\in[d]\} \foral F\subseteq[d].
\]
Then, for each $F\subseteq[d]$ and $a\in A$, we have that $\lim_{\un{n}\perp F}\|\phi_\un{n}(a)+\K(X_\un{n}\L_F)\|=0$ if and only if for each $\varepsilon>0$ there exists $\un{n}\perp F$ and $k_\un{n} \in\K(X_\un{n}\L_F)$ such that $\|\phi_\un{n}(a)+k_\un{n}\|<\varepsilon$.
\end{proposition}

When $\L\subseteq\I$ is invariant, partially ordered and consists of ideals, the same is true of $\L^{(1)}$ and moreover $\L\subseteq\L^{(1)}$ and $\fJ_\L=\fJ_{\L^{(1)}}$ \cite[Proposition 3.4.5]{DeK24}.
Maximality of $\L$ is realised exactly when we have the reverse containment $\L^{(1)}\subseteq\L$.

\begin{theorem}\label{T:m fam v2}\cite[Theorem 3.4.6]{DeK24}
Let $X$ be a strong compactly aligned product system with coefficients in a C*-algebra $A$ and suppose that $\L$ is a $2^d$-tuple of $X$ satisfying $\L\subseteq\I$.
Then $\L$ is maximal if and only if it satisfies the following four conditions:
\begin{enumerate}
\item $\L$ consists of ideals,
\item $\L$ is invariant,
\item $\L$ is partially ordered,
\item $\L^{(1)} \subseteq \L$.
\end{enumerate}
\end{theorem}

Dealing with the ``non-injective" case now follows a similar trajectory to the preceding reasoning, but argues on the level of quotient product systems instead.
Intuitively, to parametrise the gauge-invariant ideals $\fJ$ of $\N\T_X$ such that $(Q_\fJ\circ\ol{\pi}_X,Q_\fJ\circ\ol{t}_X)$ is non-injective, we quotient out $\ker Q_\fJ\circ\ol{\pi}_X$ to obtain an injective Nica-covariant representation of a certain quotient product system.
This representation inherits the gauge action of $(Q_\fJ\circ\ol{\pi}_X,Q_\fJ\circ\ol{t}_X)$ \cite[Lemma 4.1.9 (iii)]{DeK24} and so we can exploit the work that we have already done to complete the parametrisation.
This ethos underpins \cite[Section 4]{DeK24}, whose key results we now summarise.

\begin{definition}\label{D:JF}\cite[Definition 4.1.1]{DeK24}
Let $X$ be a strong compactly aligned product system with coefficients in a C*-algebra $A$.
Fix $\mt\neq F\subseteq[d]$ and let $I\subseteq A$ be an ideal. 
We define the following subsets of $A$:
\begin{enumerate}
\item $X_F^{-1}(I):=\bigcap\{X_\un{n}^{-1}(I)\mid \un{0}\neq\un{n}\leq\un{1}_F\} = \{a \in A \mid \sca{X_{\un{n}}, a X_{\un{n}}} \subseteq I \foral \un{0} \neq \un{n} \leq \un{1}_F\}$,
\item $J_F(I,X):=\{a\in A\mid [\phi_\un{i}(a)]_I\in\K([X_\un{i}]_I)\foral i\in[d], aX_F^{-1}(I)\subseteq I\}$.
\end{enumerate}
\end{definition}

Both $X_F^{-1}(I)$ and $J_F(I,X)$ are ideals of $A$, and $I\subseteq J_F(I,X)$ whenever $I$ is positively invariant.
These objects play similar roles to the ideals $X^{-1}(I)$ and $J(I,X)$ for a C*-correspondence $X$ (see the discussion preceding Theorem \ref{T:Kat par}).
When $I$ is positively invariant, the ideals $X_F^{-1}(I)$ and $J_F(I,X)$ relate to the product system structure of $[X]_I$ (which is itself strong compactly aligned by Proposition \ref{P:qntrprodsys}) in the following sense.

\begin{proposition}\label{P:quo ideal}\cite[Lemma 4.1.3]{DeK24}
Let $X$ be a strong compactly aligned product system with coefficients in a C*-algebra $A$ and let $I\subseteq A$ be an ideal that is positively invariant for $X$.
Then the following hold for all $\mt\neq F\subseteq[d]$:
\begin{enumerate}
\item $X_F^{-1}(I)=[ \hspace{1pt} \cdot \hspace{1pt} ]_I^{-1}(\bigcap\{\ker[\phi_\un{i}]_I\mid i\in F\})$.
\item $J_F(I,X)=[ \hspace{1pt} \cdot \hspace{1pt} ]_I^{-1}(\J_F([X]_I))$.
\item $X_F^{-1}(I)\cap J_F(I,X)=I$.
\end{enumerate}
\end{proposition}

With these objects in hand, we are ready to define the parametrising objects of \cite{DeK24}.

\begin{definition}\label{D:NT tuple}\cite[Definition 4.1.4]{DeK24}
Let $X$ be a strong compactly aligned product system with coefficients in a C*-algebra $A$ and let $\L$ be a $2^d$-tuple of $X$.
We say that $\L$ is an \emph{NT-$2^d$-tuple (of $X$)} if the following four conditions hold:
\begin{enumerate}
\item $\L$ consists of ideals and $\L_F\subseteq J_F(\L_\mt,X)$ for all $\mt\neq F\subseteq[d]$,
\item $\L$ is $X$-invariant,
\item $\L$ is partially ordered,
\item $ [ \hspace{1pt} \cdot \hspace{1pt} ]_{\L_\mt}^{-1} \big( [\L_F]_{\L_\mt}^{(1)} \big) \subseteq \L_F$ for all $F\subseteq[d]$, where $[\L_F]_{\L_\mt}=\L_F/\L_\mt\subseteq [A]_{\L_\mt}$.
\end{enumerate}
\end{definition}

To make sense of condition (iv), first note that conditions (i) and (ii) imply that $\L_\mt$ is an ideal of $A$ that is positively invariant for $X$.
Hence we can make sense of $[X]_{\L_\mt}$ as a strong compactly aligned product system with coefficients in $[A]_{\L_\mt}$ by Proposition \ref{P:qntrprodsys}.
Condition (iii) implies that $\L_\mt\subseteq\L_F$ for all $F\subseteq[d]$, and so by condition (i) we have that $[\L]_{\L_\mt}:=\{[\L_F]_{\L_\mt}\}_{F\subseteq[d]}$ is a $2^d$-tuple of $[X]_{\L_\mt}$ that consists of ideals.
Applying condition (i) and Proposition \ref{P:quo ideal} in tandem gives that $[\L]_{\L_\mt}\subseteq\J([X]_{\L_\mt})$, while condition (ii) implies that $[\L]_{\L_\mt}$ is $[X]_{\L_\mt}$-invariant.
Hence we have that $[\L]_{\L_\mt}\subseteq\I([X]_{\L_\mt})$ by \cite[Lemma 3.2.3]{DeK24}, and so we can consider the family $[\L]_{\L_\mt}^{(1)}$.
Note also that condition (iv) holds automatically for $F= \mt$ and $F=[d]$.

When $X$ is proper, condition (iv) admits the following simplification.

\begin{proposition}\label{P:NTcom}\cite[Proposition 4.1.5]{DeK24}
Let $X$ be a proper product system over $\bZ_+^d$ with coefficients in a C*-algebra $A$.
Then a $2^d$-tuple $\L$ of $X$ is an NT-$2^d$-tuple of $X$ if and only if it satisfies conditions (i)-(iii) of Definition \ref{D:NT tuple}
and
\[
\bigg(\bigcap_{\un{n}\perp F}X_\un{n}^{-1}(J_F(\L_\mt,X))\bigg)\cap\L_{\inv,F}\cap\L_{\lim,F}\subseteq\L_F\foral\mt\neq F\subsetneq[d].
\]
\end{proposition}

As advertised, the NT-$2^d$-tuples of $X$ parametrise the gauge-invariant ideals of $\N\T_X$.
More precisely, we have the following result.

\begin{theorem}\label{T:NT param}\cite[Theorem 4.2.3]{DeK24}
Let $X$ be a strong compactly aligned product system with coefficients in a C*-algebra $A$.
Then there exists an order-preserving bijection between the set of NT-$2^d$-tuples of $X$ and the set of gauge-invariant ideals of $\N\T_X$.
\end{theorem}

It should be noted that Theorem \ref{T:d GIUT M} is used in the proof of Theorem \ref{T:NT param}.
The lattice operations on the set of NT-$2^d$-tuples that promote the bijection of Theorem \ref{T:NT param} to a lattice isomorphism are clarified in \cite[Propositions 4.2.6 and 4.2.7]{DeK24}.
With minor alterations, Theorem \ref{T:NT param} can be modified to parametrise the gauge-invariant ideals of $\N\O(\K,X)$ for any relative $2^d$-tuple $\K$ of $X$.
In particular, we obtain from this a parametrisation of the gauge-invariant ideals of $\N\O_X$ by taking $\K=\I$.

\begin{definition}\label{D:rel NO tuple}\cite[Definition 4.2.8]{DeK24}
Let $X$ be a strong compactly aligned product system with coefficients in a C*-algebra $A$.
Let $\K$ be a relative $2^d$-tuple of $X$ and let $\L$ be a  $2^d$-tuple of $X$.
We say that $\L$ is a \emph{$\K$-relative NO-$2^d$-tuple (of $X$)} if $\L$ is an NT-$2^d$-tuple of $X$ and $\K \subseteq \L$.
We refer to the $\I$-relative NO-$2^d$-tuples of $X$ simply as \emph{NO-$2^d$-tuples (of $X$)}.
\end{definition}

The lattice operations on the set of NT-$2^d$-tuples restrict appropriately to the set of $\K$-relative NO-$2^d$-tuples for an arbitrary relative $2^d$-tuple $\K$ \cite[Proposition 4.2.10]{DeK24}.
With this, we obtain the main parametrisation result of \cite{DeK24} at full generality.

\begin{theorem}\label{T:NO param}\cite[Theorem 4.2.11]{DeK24}
Let $X$ be a strong compactly aligned product system with coefficients in a C*-algebra $A$ and let $\K$ be a relative $2^d$-tuple of $X$.
Then there exists an order-preserving bijection between the set of $\K$-relative NO-$2^d$-tuples of $X$ and the set of gauge-invariant ideals of $\N\O(\K,X)$.
\end{theorem}

The bijection of Theorem \ref{T:NO param} is bolstered to a lattice isomorphism by equipping the set of $\K$-relative NO-$2^d$-tuples with the lattice structure mentioned previously.
Theorem \ref{T:NO param} completely describes the gauge-invariant ideal structure of every equivariant quotient of $\N\T_X$, since every such quotient is canonically $*$-isomorphic to a relative Cuntz-Nica-Pimsner algebra (not necessarily of $X$, but certainly of a quotient of $X$) \cite[Proposition 4.2.1]{DeK24}.

We close this subsection by clarifying the relationship between maximal $2^d$-tuples and NT-$2^d$-tuples in the case where $X$ is proper.

\begin{proposition}\label{P:maxNT}
Let $X$ be a proper product system over $\bZ_+^d$ with coefficients in a C*-algebra $A$ and suppose that $\L$ is a $2^d$-tuple of $X$.
Then the following are equivalent:
\begin{enumerate}
\item $\L$ is a maximal $2^d$-tuple of $X$;
\item $\L=\L^{(\pi,t)}$ for some Nica-covariant representation $(\pi,t)$ of $X$ that admits a gauge action;
\item $\L$ is an NT-$2^d$-tuple of $X$.
\end{enumerate}
\end{proposition}
\begin{proof}
The equivalence of (ii) and (iii) follows by \cite[Proposition 4.1.12]{DeK24}, so it suffices to show the equivalence of (i) and (ii).

Assume that $\L$ is a maximal $2^d$-tuple of $X$ and set $\fJ:=\fJ_\L$.
Let $Q_\fJ\colon\N\T_X\to\N\T_X/\fJ$ denote the quotient map.
Recall that $(Q_\fJ\circ\ol{\pi}_X,Q_\fJ\circ\ol{t}_X)$ is a Nica-covariant representation of $X$ that admits a gauge action.
It is routine to check that the induced $*$-representation of $\K(X_\un{n})$ is $Q_\fJ\circ\ol{\psi}_{X,\un{n}}$ for each $\un{n}\in\bZ_+^d$.
It suffices to show that $\L=\L^{(Q_\fJ\circ\ol{\pi}_X,Q_\fJ\circ\ol{t}_X)}$.
We start by showing that $\L\subseteq\L^{(Q_\fJ\circ\ol{\pi}_X,Q_\fJ\circ\ol{t}_X)}$.
To this end, it is instructive to recall the definition of the ideal $\fJ$:
\[
\fJ\equiv\sca{\ol{\pi}_X(\L_F)\ol{q}_{X,F}\mid F\subseteq[d]}.
\]
Thus we have that
\[
\ol{\pi}_X(\L_\mt)=\ol{\pi}_X(\L_\mt)\ol{q}_{X,\mt}\subseteq\fJ,
\]
from which it follows that $\L_\mt\subseteq\L^{(Q_\fJ\circ\ol{\pi}_X,Q_\fJ\circ\ol{t}_X)}_\mt\equiv\ker Q_\fJ\circ\ol{\pi}_X$.
Now fix $\mt\neq F\subseteq[d]$ and $a\in\L_F$.
An application of Proposition \ref{P:prod cai} yields that
\[
\fJ\ni\ol{\pi}_X(a)\ol{q}_{X,F}=\ol{\pi}_X(a)+\sum\{(-1)^{|\un{n}|}\ol{\psi}_{X,\un{n}}(\phi_\un{n}(a))\mid\un{0}\neq\un{n}\leq\un{1}_F\}
\]
and hence
\[
Q_\fJ(\ol{\pi}_X(a))+\sum\{(-1)^{|\un{n}|}Q_\fJ(\ol{\psi}_{X,\un{n}}(\phi_\un{n}(a)))\mid\un{0}\neq\un{n}\leq\un{1}_F\}=Q_\fJ(\ol{\pi}_X(a)\ol{q}_{X,F})=0.
\]
It follows that $a\in\L^{(Q_\fJ\circ\ol{\pi}_X,Q_\fJ\circ\ol{t}_X)}_F\equiv(Q_\fJ\circ\ol{\pi}_X)^{-1}(B_{(\un{0},\un{1}_F]}^{(Q_\fJ\circ\ol{\pi}_X,Q_\fJ\circ\ol{t}_X)})$ and so $\L\subseteq\L^{(Q_\fJ\circ\ol{\pi}_X,Q_\fJ\circ\ol{t}_X)}$.
For the reverse inclusion, note that properness of $X$ ensures that $\L^{(Q_\fJ\circ\ol{\pi}_X,Q_\fJ\circ\ol{t}_X)}$ is a relative $2^d$-tuple. 
Thus maximality of $\L$ implies that it is sufficient to show that
\[
\fJ=\fJ_{\L^{(Q_\fJ\circ\ol{\pi}_X,Q_\fJ\circ\ol{t}_X)}}\equiv\sca{\ol{\pi}_X(\L^{(Q_\fJ\circ\ol{\pi}_X,Q_\fJ\circ\ol{t}_X)}_F)\ol{q}_{X,F}\mid F\subseteq[d]}.
\]
To this end, observe that the forward inclusion is immediate since $\L\subseteq\L^{(Q_\fJ\circ\ol{\pi}_X,Q_\fJ\circ\ol{t}_X)}$.
For the reverse inclusion, fix $F\subseteq[d]$ and $a\in\L_F^{(Q_\fJ\circ\ol{\pi}_X,Q_\fJ\circ\ol{t}_X)}$.
It suffices to show that $\ol{\pi}_X(a)\ol{q}_{X,F}\in\fJ$.
This is immediate when $F=\mt$, so we may assume that $F\neq\mt$ without loss of generality.
Since $a\in\L_F^{(Q_\fJ\circ\ol{\pi}_X,Q_\fJ\circ\ol{t}_X)}$, by definition there exists $k_\un{n}\in\K(X_\un{n})$ for each $\un{0}\neq\un{n}\leq\un{1}_F$ such that
\[
Q_\fJ(\ol{\pi}_X(a))=\sum\{Q_\fJ(\ol{\psi}_{X,\un{n}}(k_\un{n}))\mid\un{0}\neq\un{n}\leq\un{1}_F\}.
\]
Hence we obtain that
\[
Q_\fJ(\ol{\pi}_X(a)\ol{q}_{X,F})=Q_\fJ(\ol{\pi}_X(a))+\sum\{(-1)^{|\un{n}|}Q_\fJ(\ol{\psi}_{X,\un{n}}(\phi_\un{n}(a)))\mid\un{0}\neq\un{n}\leq\un{1}_F\}=0,
\]
where the first equality follows by Proposition \ref{P:prod cai} and the second follows by a combination of Propositions \ref{P:prod cai} and \ref{P:DK3.3}, replacing ``$(\pi,t)$" by ``$(Q_\fJ\circ\ol{\pi}_X,Q_\fJ\circ\ol{t}_X)$" in the statements of both.
Thus $\ol{\pi}_X(a)\ol{q}_{X,F}\in\fJ$, as required.

For the converse, assume that $\L=\L^{(\pi,t)}$ for some Nica-covariant representation $(\pi,t)$ of $X$ that admits a gauge action.
Let $\L'$ be a $2^d$-tuple of $X$ (which is automatically relative by properness of $X$) such that $\L\subseteq\L'$ and $\fJ_\L=\fJ_{\L'}$.
We must show that $\L'\subseteq\L$.
To this end, fix $F\subseteq[d]$ and $a\in\L_F'$.
By definition we have that $\ol{\pi}_X(a)\ol{q}_{X,F}\in\fJ_{\L'}=\fJ_\L$.
Observe that
\[
(\pi\times t)(\fJ_\L)=\sca{\pi(\L_D)q_D\mid D\subseteq[d]}
\]
by canonicity of $\pi\times t\colon\N\T_X\to\ca(\pi,t)$.
In turn, we have that
\[
\pi(a)q_F=(\pi\times t)(\ol{\pi}_X(a)\ol{q}_{X,F})\in\sca{\pi(\L_D)q_D\mid D\subseteq[d]}
\]
by Proposition \ref{P:prod cai}.
However, since $\L=\L^{(\pi,t)}$ by assumption, we obtain that
\[
\sca{\pi(\L_D)q_D\mid D\subseteq[d]}=\sca{\pi(\L_D^{(\pi,t)})q_D\mid D\subseteq[d]}=\{0\},
\]
where the final equality follows by Proposition \ref{P:DK3.3}.
Thus $\pi(a)q_F=0$ and so $a\in\L_F^{(\pi,t)}=\L_F$ by Proposition \ref{P:prod cai}.
Hence $\L'\subseteq\L$ and we conclude that $\L$ is maximal, finishing the proof.
\end{proof}

\subsection{T-families}\label{Ss:T fam}

Next we present the key tools used to achieve the parametrisation result of \cite{BB24}.
Much of the work therein focuses on an arbitrary proper product system $X$ over $\bZ_+^d$ with coefficients in a C*-algebra $A$, so we will restrict our attention to this setting throughout the subsection.
The approach adopted in \cite{BB24} makes use of an extended product system construction \cite[Section 4.2]{BB24} and thus differs from \cite{DeK24} quite substantially.
Nevertheless, there are some key commonalities, including the use of relative Cuntz-Nica-Pimsner algebras \cite[Section 4.3]{BB24} and a Gauge-Invariant Uniqueness Theorem \cite[Corollary 4.14]{BB24}.

Recalling that $X$ is automatically strong compactly aligned by Proposition \ref{P:propimpsca}, we begin by addressing how some of the machinery covered up to this point simplifies in the proper case.
Firstly, every $2^d$-tuple of $X$ is automatically relative.
Next, given a Nica-covariant representation $(\pi,t)$ of $X$, we have that 
\[
\pi(a)q_F=\pi(a)+\sum\{(-1)^{|\un{n}|}\psi_\un{n}(\phi_\un{n}(a))\mid\un{0}\neq\un{n}\leq\un{1}_F\}\foral a\in A \; \text{and} \; \mt\neq F\subseteq[d]
\]
by Proposition \ref{P:prod cai}.
In turn, fixing $a\in A$ and $F\subseteq[d]$, we have that
\begin{equation}\label{Eq:proptrick}
\pi(a)q_F=0\iff a\in\L_F^{(\pi,t)},
\end{equation}
where the reverse implication follows by Proposition \ref{P:DK3.3}.
Lastly, fixing an ideal $I\subseteq A$, we deduce that
\[
\J_F=(\bigcap_{i\in F}\ker\phi_\un{i})^\perp\qand J_F(I,X)=\{a\in A\mid aX_F^{-1}(I)\subseteq I\}\foral\mt\neq F\subseteq[d],
\]
where the simplification of $J_F(I,X)$ follows by Lemma \ref{L:Kat07}.

\begin{definition}\label{D:T+O}\cite[Definition 4.2]{BB24}
Let $X$ be a proper product system over $\bZ_+^d$ with coefficients in a C*-algebra $A$.
A $2^d$-tuple $\L$ of $X$ is a \emph{T-family (of $X$)} if it consists of ideals and satisfies
\begin{equation}\label{Eq:Tfam}
\L_F=X_\un{i}^{-1}(\L_F)\cap\L_{F\cup\{i\}}\foral F\subsetneq[d] \; \text{and} \; i\in[d]\setminus F.
\end{equation}
A T-family $\L$ of $X$ is said to be an \emph{O-family (of $X$)} if $\I\subseteq\L$.
\end{definition}

Related to T-families are the \emph{invariant families} \cite[Definition 4.1]{BB24} (not to be confused with item (i) of Definition \ref{D:inv po}).
The set of invariant families of $X$ is in order-preserving bijection with the set of T-families of $X$ \cite[Proposition 4.4]{BB24} and thus we focus our attention on the latter.
T-families admit the following Gauge-Invariant Uniqueness Theorem.

\begin{theorem}\label{T:tfamGIUT}\cite[Corollary 4.14]{BB24}
Let $X$ be a proper product system over $\bZ_+^d$ with coefficients in a C*-algebra $A$.
Let $\L$ be a T-family of $X$ and $(\pi,t)$ be an $\L$-relative CNP-representation of $X$.
Then $\N\O(\L,X)\cong\ca(\pi,t)$ via a (unique) canonical $*$-isomorphism if and only if $\pi(a)q_F=0$ implies that $a\in\L_F$ (for all $a\in A$ and $F\subseteq[d]$) and $(\pi,t)$ admits a gauge action. 
\end{theorem}

It should be noted that the set of T-families of $X$ and the set of $2^d$-tuples $\L$ of $X$ that satisfy $\L\subseteq\I$ are not comparable.
In other words, there exist T-families $\L$ that do not satisfy $\L\subseteq\I$, as well as $2^d$-tuples $\L$ which satisfy $\L\subseteq\I$ but not (\ref{Eq:Tfam}).
Accordingly, Theorem \ref{T:tfamGIUT} and \cite[Theorem 3.4.9]{DeK24} should not be conflated, even upon restriction of the latter to the proper case.
We will instantiate this in Subsection \ref{Ss:c*-dynsys} (see Remark \ref{R:GIUTdiff}), as we will require a product system construction arising from the theory of C*-dynamical systems.

It will be useful to rephrase Theorem \ref{T:tfamGIUT} via the following lemma.

\begin{lemma}\label{L:GIUTequiv}
Let $X$ be a proper product system over $\bZ_+^d$ with coefficients in a C*-algebra $A$.
Let $\L$ be a T-family of $X$ and $(\pi,t)$ be an $\L$-relative CNP-representation of $X$.
Then $\pi(a)q_F=0$ implies that $a\in\L_F$ (for all $a\in A$ and $F\subseteq[d]$) if and only if $\L=\L^{(\pi,t)}$.
\end{lemma}
\begin{proof}
Immediate by (\ref{Eq:proptrick}) and the fact that $(\pi,t)$ is an $\L$-relative CNP-representation.
\end{proof}

The main result of \cite{BB24} is as follows.

\begin{theorem}\label{T:Bparam}\cite[Theorem 4.15]{BB24}
Let $X$ be a proper product system over $\bZ_+^d$ with coefficients in a C*-algebra $A$.
Then the following hold:
\begin{enumerate}
\item there exists an order-preserving bijection between the set of T-families of $X$ and the set of gauge-invariant ideals of $\N\T_X$, and
\item there exists an order-preserving bijection between the set of O-families of $X$ and the set of gauge-invariant ideals of $\N\O_X$.
\end{enumerate}
\end{theorem}

Theorems \ref{T:NO param} and \ref{T:Bparam} both clarify the gauge-invariant ideal structure of $\N\T_X$ and $\N\O_X$: the former on the level of strong compactly aligned product systems, and the latter on the level of proper product systems over $\bZ_+^d$.
In both cases the parametrising objects are subfamilies of $2^d$-tuples of $X$ and both proofs make use of a Gauge-Invariant Uniqueness Theorem.
With these similarities in mind, it is now natural to ask if the parametrising objects of the two theorems are in fact the same on the level of proper product systems over $\bZ_+^d$.
Answering this question in the affirmative will be our primary focus going forward.

Provided that one is willing to take a detour via Nica-covariant representations and the Gauge-Invariant Uniqueness Theorems of \cite{BB24, DeK24}, arriving at the aforementioned answer is reasonably straightforward.
The following two results demonstrate how to achieve this.

\begin{proposition}\label{P:repTfam}
Let $X$ be a proper product system over $\bZ_+^d$ with coefficients in a C*-algebra $A$ and let $(\pi,t)$ be a Nica-covariant representation of $X$.
Then $\L^{(\pi,t)}$ is a T-family of $X$.
\end{proposition}
\begin{proof}
We have already remarked that $\L^{(\pi,t)}$ consists of ideals.
Thus, fixing $F\subsetneq[d]$ and $i\in[d]\setminus F$, it suffices to show that
\[
\L_F^{(\pi,t)}=X_\un{i}^{-1}(\L_F^{(\pi,t)})\cap\L_{F\cup\{i\}}^{(\pi,t)}.
\]
The forward inclusion is immediate since $\L^{(\pi,t)}$ is invariant and partially ordered by \cite[Proposition 3.1.14]{DeK24}.
For the reverse inclusion, fix $a\in X_\un{i}^{-1}(\L_F^{(\pi,t)})\cap\L_{F\cup\{i\}}^{(\pi,t)}$.
Applying (\ref{Eq:proptrick}) twice, we obtain that
\[
t_\un{i}(X_\un{i})^*\pi(a)q_Ft_\un{i}(X_\un{i})=t_\un{i}(X_\un{i})^*\pi(a)t_\un{i}(X_\un{i})q_F=\pi(\sca{X_\un{i},aX_\un{i}})q_F=\{0\}\;\text{and}\;\pi(a)q_{F\cup\{i\}}=0,
\]
where we also use Proposition \ref{P:pf reducing} in the first equality.
In particular, it follows that
\[
\psi_\un{i}(\K(X_\un{i}))\pi(a)q_F\psi_\un{i}(\K(X_\un{i}))=\{0\},
\]
since $\psi_\un{i}(\K(X_\un{i}))=[t_\un{i}(X_\un{i})t_\un{i}(X_\un{i})^*]$.
Fixing $k_\un{i},k_\un{i}'\in\K(X_\un{i})$ and writing $\pi(a)q_F$ as an alternating sum using Proposition \ref{P:prod cai}, we deduce that
\begin{equation}\label{Eq:altsumzero}
\psi_\un{i}(k_\un{i})\pi(a)\psi_\un{i}(k_\un{i}')+\sum\{(-1)^{|\un{n}|}\psi_\un{i}(k_\un{i})\psi_\un{n}(\phi_\un{n}(a))\psi_\un{i}(k_\un{i}')\mid\un{0}\neq\un{n}\leq\un{1}_F\}=0.
\end{equation}
Note that we take the $\Sigma$-summand to be $0$ when $F=\mt$.
For $\un{0}\neq\un{n}\leq\un{1}_F$, we have that
\[
\psi_\un{i}(k_\un{i})\psi_\un{n}(\phi_\un{n}(a))\psi_\un{i}(k_\un{i}')=\psi_{\un{n}+\un{i}}(\iota_\un{i}^{\un{n}+\un{i}}(k_\un{i})\iota_\un{n}^{\un{n}+\un{i}}(\phi_\un{n}(a))\iota_\un{i}^{\un{n}+\un{i}}(k_\un{i}'))
\]
by Nica-covariance of $(\pi,t)$, noting that $\un{n}\vee\un{i}=\un{n}+\un{i}$ since $\un{n}\perp\un{i}$.
As $k_\un{i},k_\un{i}'\in\K(X_\un{i})$ are arbitrary, we may replace them by members of an approximate unit $(k_{\un{i},\la})_{\la\in\La}$ of $\K(X_\un{i})$ and use Proposition \ref{P:sca ai} (taking $F=\{i\}$ therein) to obtain that
\[
\nor{\cdot}\text{-}\lim_\la\psi_\un{i}(k_{\un{i},\la})\psi_\un{n}(\phi_\un{n}(a))\psi_\un{i}(k_{\un{i},\la})=\psi_{\un{n}+\un{i}}(\phi_{\un{n}+\un{i}}(a))\foral\un{n}\leq\un{1}_F,
\]
noting that $\psi_\un{0}(\phi_\un{0}(a))=\pi(a)$ and that $\iota_\un{n}^{\un{n}+\un{i}}(\phi_\un{n}(a))=\phi_{\un{n}+\un{i}}(a)$ for all $\un{n}\leq\un{1}_F$.
Combining this with (\ref{Eq:altsumzero}), we deduce that
\begin{equation}\label{Eq:ilift}
\psi_\un{i}(\phi_\un{i}(a))+\sum\{(-1)^{|\un{n}|}\psi_{\un{n}+\un{i}}(\phi_{\un{n}+\un{i}}(a))\mid\un{0}\neq\un{n}\leq\un{1}_F\}=0.
\end{equation}
Recalling that $\pi(a)q_{F\cup\{i\}}=0$, another application of Proposition \ref{P:prod cai} gives that
\begin{equation}\label{Eq:Fcupireln}
\pi(a)+\sum\{(-1)^{|\un{n}|}\psi_\un{n}(\phi_\un{n}(a))\mid\un{0}\neq\un{n}\leq\un{1}_{F\cup\{i\}}\}=0.
\end{equation}
By summing (\ref{Eq:ilift}) and (\ref{Eq:Fcupireln}), we deduce that
\[
\pi(a)+\sum\{(-1)^{|\un{n}|}\psi_\un{n}(\phi_\un{n}(a))\mid\un{0}\neq\un{n}\leq\un{1}_F\}=0.
\]
It follows that $a\in\L_F^{(\pi,t)}$, completing the proof.
\end{proof}

\begin{proposition}\label{P:Tfamga}
Let $X$ be a proper product system over $\bZ_+^d$ with coefficients in a C*-algebra $A$ and suppose that $\L$ is a $2^d$-tuple of $X$.
Then $\L=\L^{(\pi,t)}$ for some Nica-covariant representation $(\pi,t)$ of $X$ that admits a gauge action if and only if $\L$ is a T-family of $X$.
\end{proposition}
\begin{proof}
The forward implication follows by Proposition \ref{P:repTfam}.
For the converse, assume that $\L$ is a T-family of $X$.
Let $(\pi_X^\L,t_X^\L)$ denote the universal $\L$-relative CNP-representation of $X$.
Since $\N\O(\L,X)$ is canonically $*$-isomorphic to itself via the identity map, combining Theorem \ref{T:tfamGIUT} and Lemma \ref{L:GIUTequiv} yields that $\L=\L^{(\pi_X^\L,t_X^\L)}$ and that $(\pi_X^\L,t_X^\L)$ admits a gauge action.
This completes the proof.
\end{proof}

Combining Propositions \ref{P:maxNT} and \ref{P:Tfamga}, we obtain the promised alignment of NT-$2^d$-tuples with T-families, since both sets of objects coincide with the set of $2^d$-tuples of the form $\L^{(\pi,t)}$ for a Nica-covariant representation $(\pi,t)$ that admits a gauge action.
The alignment of NO-$2^d$-tuples with O-families follows as an immediate consequence.

\section{Connection between NT-$2^d$-tuples and T-families}\label{S:NT=T}

\noindent A shortcoming of the argument provided at the end of Section \ref{S:BBDeK} lies in its indirectness, i.e., it requires a strong understanding of the technical machinery of \cite{BB24, DeK24} (e.g., two Gauge-Invariant Uniqueness Theorems are used in the proof).
To remedy this, we will instead seek to establish the alignment of NT-$2^d$-tuples with T-families directly, using the definitions alone.
In this way, we will be able to eschew any and all discussion of Nica-covariant representations, gauge actions and Gauge-Invariant Uniqueness Theorems.

Throughout this section we take $X$ to be a proper product system over $\bZ_+^d$ with coefficients in a C*-algebra $A$.
To show the alignment of NT-$2^d$-tuples with T-families, we first show that every NT-$2^d$-tuple is a T-family and then that every T-family is an NT-$2^d$-tuple.
Both directions will require some auxiliary results, so they are given their own subsections.

\subsection{NT-$2^d$-tuples are T-families}\label{Ss:NTimpliesT}

We commence with a proposition that generalises \cite[Lemma 4.3.4]{DFK17} from the context of C*-dynamical systems to the context of strong compactly aligned product systems.

\begin{proposition}\label{P:Jpassdown}
Let $X$ be a strong compactly aligned product system with coefficients in a C*-algebra $A$.
Then we have that 
\[
X_\un{i}^{-1}(\J_F)\cap\J_{F\cup\{i\}}\subseteq\J_F\foral F\subsetneq[d] \; \text{and} \; i\in[d]\setminus F.
\]
\end{proposition}
\begin{proof}
First we prove the claim for $F=\mt$.
To this end, take $i\in[d]$ and $a\in X_\un{i}^{-1}(\J_\mt)\cap\J_{\{i\}}$.
By definition of $X_\un{i}^{-1}(\J_\mt)$, we have that
\[
\sca{X_\un{i},aX_\un{i}}\subseteq\J_\mt=\{0\}
\]
and thus $a\in\ker\phi_\un{i}$.
Since $a\in\J_{\{i\}}\subseteq(\ker\phi_\un{i})^\perp$, it follows that $a=0$, as required.

Now fix $\mt\neq F\subsetneq[d], i\in[d]\setminus F$ and $a\in X_\un{i}^{-1}(\J_F)\cap\J_{F\cup\{i\}}$.
Then by definition we have that
\[
\sca{X_\un{i},aX_\un{i}}\subseteq(\bigcap_{j\in F}\ker\phi_\un{j})^\perp\cap(\bigcap_{j\in[d]}\phi_\un{j}^{-1}(\K(X_\un{j})))\;\text{and}\; a\in(\bigcap_{j\in F\cup\{i\}}\ker\phi_\un{j})^\perp\cap(\bigcap_{j\in[d]}\phi_\un{j}^{-1}(\K(X_\un{j}))).
\]
Showing that $a\in\J_F$ amounts to proving that $a\in(\bigcap_{j\in F}\ker\phi_\un{j})^\perp$, so fix $b\in\bigcap_{j\in F}\ker\phi_\un{j}$.
We claim that 
\[
\sca{X_\un{i},bX_\un{i}}\subseteq\bigcap_{j\in F}\ker\phi_\un{j}.
\]
To see this, fix $\xi_\un{i},\eta_\un{i}\in X_\un{i}$ and $j\in F$.
We define operators $\tau(\xi_\un{i}),\tau(\eta_\un{i})\in\L(X_\un{j},X_{\un{j}+\un{i}})$ by
\[
\tau(\xi_\un{i}):=u_{\un{i},\un{j}}\circ\Theta_{\xi_\un{i}}\qand\tau(\eta_\un{i}):=u_{\un{i},\un{j}}\circ\Theta_{\eta_\un{i}},
\]
where the operators $\Theta_{\xi_\un{i}},\Theta_{\eta_\un{i}}\in\L(X_\un{j},X_\un{i}\otimes_A X_\un{j})$ are defined as in Lemma \ref{L:lance} (taking $X=X_\un{i}$ and $Y=X_\un{j}$ therein).
It is routine to check that
\[
\phi_\un{j}(\sca{\xi_\un{i},b\eta_\un{i}})=\tau(\xi_\un{i})^*\phi_{\un{j}+\un{i}}(b)\tau(\eta_\un{i}).
\]
In turn, noting that $b\in\ker\phi_\un{j}$, we obtain that
\[
\tau(\xi_\un{i})^*\phi_{\un{j}+\un{i}}(b)\tau(\eta_\un{i})=\tau(\xi_\un{i})^*\iota_\un{j}^{\un{j}+\un{i}}(\phi_\un{j}(b))\tau(\eta_\un{i})=0.
\]
Since $\xi_\un{i},\eta_\un{i}\in X_\un{i}$ and $j\in F$ are arbitrary, we deduce that $\sca{X_\un{i},bX_\un{i}}\subseteq\bigcap_{j\in F}\ker\phi_\un{j}$, as claimed.
Hence we have that
\[
\sca{X_\un{i},aX_\un{i}}\sca{X_\un{i},bX_\un{i}}=\{0\}.
\]
Fixing $\xi_\un{i},\eta_\un{i},\ze_\un{i},\nu_\un{i}\in X_\un{i}$, we compute the following:
\begin{align*}
\sca{\xi_\un{i},(\phi_\un{i}(a)\Theta_{\eta_\un{i},\ze_\un{i}}\phi_\un{i}(b))\nu_\un{i}}=\sca{\xi_\un{i},a(\Theta_{\eta_\un{i},\ze_\un{i}}(b\nu_\un{i}))}=\sca{\xi_\un{i},a(\eta_\un{i}\sca{\ze_\un{i},b\nu_\un{i}})}=\sca{\xi_\un{i},a\eta_\un{i}}\sca{\ze_\un{i},b\nu_\un{i}}=0.
\end{align*}
Since $\xi_\un{i},\nu_\un{i}\in X_\un{i}$ are arbitrary, we deduce that
\[
\phi_\un{i}(a)\Theta_{\eta_\un{i},\ze_\un{i}}\phi_\un{i}(b)=0\foral\eta_\un{i},\ze_\un{i}\in X_\un{i}.
\]
In turn, because $\eta_\un{i},\ze_\un{i}\in X_\un{i}$ are arbitrary, it follows that
\[
\phi_\un{i}(a)\K(X_\un{i})\phi_\un{i}(b)=\{0\}.
\]
Since $\phi_\un{i}(a)\in\K(X_\un{i})$, an application of an approximate unit of $\K(X_\un{i})$ gives that $\phi_\un{i}(ab)=0$ and hence $ab\in\bigcap_{j\in F\cup\{i\}}\ker\phi_\un{j}$.
However, we also have that $ab\in(\bigcap_{j\in F\cup\{i\}}\ker\phi_\un{j})^\perp$ since $a\in\J_{F\cup\{i\}}$.
Thus $ab=0$ and so $a\in(\bigcap_{j\in F}\ker\phi_\un{j})^\perp$, completing the proof.
\end{proof}

It should be noted that $\J$ is not a T-family in general.
To instantiate this, we will require a product system construction residing in the theory of C*-dynamical systems.
Accordingly, we defer the provision of a counterexample until Subsection \ref{Ss:c*-dynsys} (see Remark \ref{R:Jnotinv}).

We are now ready to prove that all NT-$2^d$-tuples are T-families.



\begin{proposition}\label{P:NTimpliesT}
Let $X$ be a proper product system over $\bZ_+^d$ with coefficients in a C*-algebra $A$.
Then every NT-$2^d$-tuple of $X$ is a T-family of $X$.
\end{proposition}
\begin{proof}
Let $\L$ be an NT-$2^d$-tuple of $X$.
Then $\L$ consists of ideals by item (i) of Definition \ref{D:NT tuple}.
It remains to check that $\L$ satisfies (\ref{Eq:Tfam}).
In other words, we must verify that
\begin{equation*}
\L_F=X_\un{i}^{-1}(\L_F)\cap\L_{F\cup\{i\}}\foral F\subsetneq[d] \; \text{and} \; i\in[d]\setminus F.
\end{equation*}

We begin by addressing the case where $F=\mt$.
Fixing $i\in[d]$, note that $\L_\mt\subseteq X_\un{i}^{-1}(\L_\mt)$ since $\L$ is invariant by item (ii) of Definition \ref{D:NT tuple}.
We also have that $\L_\mt\subseteq\L_{\{i\}}$ because $\L$ is partially ordered by item (iii) of Definition \ref{D:NT tuple}.
This shows that $\L_\mt\subseteq X_\un{i}^{-1}(\L_\mt)\cap\L_{\{i\}}$.
For the reverse inclusion, take $a\in X_\un{i}^{-1}(\L_\mt)\cap\L_{\{i\}}$.
An application of item (i) of Definition \ref{D:NT tuple} gives that $a\in J_{\{i\}}(\L_\mt,X)$ and hence $aX_\un{i}^{-1}(\L_\mt)\subseteq\L_\mt$.
Since $a\in X_\un{i}^{-1}(\L_\mt)$ by assumption, an application of an approximate unit of $X_\un{i}^{-1}(\L_\mt)$ yields that $a\in\L_\mt$.
Hence we have that
\[
\L_\mt=X_\un{i}^{-1}(\L_\mt)\cap\L_{\{i\}}\foral i\in[d].
\]

To account for $F\neq\mt$, we proceed by strong downward induction on $|F|$.
For the base case, fix $\mt\neq F\subsetneq[d]$ such that $|F|=d-1$ and $i\in[d]\setminus F$.
Note that $\L_F\subseteq X_\un{i}^{-1}(\L_F)\cap\L_{[d]}$ since $\L$ is invariant and partially ordered.
For the reverse inclusion, take $a\in X_\un{i}^{-1}(\L_F)\cap\L_{[d]}$.
By Proposition \ref{P:NTcom}, it suffices to show that
\[
a\in\bigg(\bigcap_{\un{n}\perp F}X_\un{n}^{-1}(J_F(\L_\mt,X))\bigg)\cap\L_{\inv,F}\cap\L_{\lim,F}.
\]
To this end, fix $\un{n}=(n_1,\dots,n_d)\perp F$.
First suppose that $n_i>0$.
Then we may write $\un{n}=\un{i}+\un{m}$ for some $\un{m}\perp F$.
Since $X_\un{i}\otimes_A X_\un{m}\cong X_\un{n}$ via the multiplication map $u_{\un{i},\un{m}}$, we obtain that
\begin{align}\label{Eq:tenstrick}
\sca{X_\un{n},aX_\un{n}}=\sca{X_\un{i}\otimes_A X_\un{m},a(X_\un{i}\otimes_A X_\un{m})}\subseteq[\sca{X_\un{m},\sca{X_\un{i},aX_\un{i}}X_\un{m}}] \subseteq\L_F\subseteq J_F(\L_\mt,X), 
\end{align}
using that $a\in X_\un{i}^{-1}(\L_F)$ and that $\L$ is invariant in the second inclusion, and item (i) of Definition \ref{D:NT tuple} in the final inclusion.
Thus $a\in X_\un{n}^{-1}(J_F(\L_\mt,X))$.
Now suppose that $n_i=0$, so that $\un{n}=\un{0}$ because $|F|=d-1$.
We must show that $a\in J_F(\L_\mt,X)$.
This is equivalent to showing that $[a]_{\L_\mt}\in\J_F([X]_{\L_\mt})$ by item (ii) of Proposition \ref{P:quo ideal}, which applies since $\L$ is invariant.
To this end, note that
\[
\sca{X_\un{i},aX_\un{i}}\subseteq\L_F\subseteq J_F(\L_\mt,X)=[\hspace{1pt}\cdot\hspace{1pt}]_{\L_\mt}^{-1}(\J_F([X]_{\L_\mt}))
\]
and that
\[
a\in\L_{[d]}\subseteq J_{[d]}(\L_\mt,X)=[\hspace{1pt}\cdot\hspace{1pt}]_{\L_\mt}^{-1}(\J_{[d]}([X]_{\L_\mt}))
\]
by assumption.
In other words, we have that
\[
[a]_{\L_\mt}\in[X_\un{i}]_{\L_\mt}^{-1}(\J_F([X]_{\L_\mt}))\cap\J_{[d]}([X]_{\L_\mt})
\]
and so $[a]_{\L_\mt}\in\J_F([X]_{\L_\mt})$ by Proposition \ref{P:Jpassdown}, which applies since $[X]_{\L_\mt}$ is proper by Lemma \ref{L:Kat07} (and so $[X]_{\L_\mt}$ is strong compactly aligned by Proposition \ref{P:qntrprodsys}).
In total, we deduce that
\[
a\in\bigcap_{\un{n}\perp F}X_\un{n}^{-1}(J_F(\L_\mt,X)).
\]

To see that $a\in\L_{\inv,F}\equiv\bigcap_{\un{n}\perp F}X_\un{n}^{-1}(\L_{[d]})$, fix $\un{n}=(n_1,\dots,n_d)\perp F$.
If $n_i>0$ then we may argue as in (\ref{Eq:tenstrick}), replacing ``$J_F(\L_\mt,X)$" by ``$\L_{[d]}$" and invoking the partial ordering of $\L$, to obtain that $a\in X_\un{n}^{-1}(\L_{[d]})$.
If $n_i=0$ (and so $\un{n}=\un{0}$), then there is nothing to show since $a\in\L_{[d]}$ by assumption.
In total, we have that $a\in\L_{\inv,F}$.

Next, since $a\in X_\un{i}^{-1}(\L_F)$ and $X$ is proper, an application of (\ref{Eq: comp}) yields that $\phi_\un{i}(a)\in\K(X_\un{i}\L_F)$.
Proposition \ref{P:lim perp} then gives that $a\in\L_{\lim,F}$.
Combining the preceding deductions, we ascertain that $a\in\L_F$, establishing the base case.

Now fix $1\leq N\leq d-2$ and suppose we have proved that (\ref{Eq:Tfam}) holds for all $\mt\neq F\subsetneq[d]$ satisfying $|F|=d-n$, for all $1\leq n\leq N$.
Fix $\mt\neq F\subsetneq[d]$ such that $|F|=d-(N+1)$ and $i\in[d]\setminus F$.
We must show that
\[
\L_F=X_\un{i}^{-1}(\L_F)\cap\L_{F\cup\{i\}}.
\]
The forward inclusion is immediate since $\L$ is invariant and partially ordered.
For the reverse inclusion, take $a\in X_\un{i}^{-1}(\L_F)\cap\L_{F\cup\{i\}}$.
As in the base case, an application of Proposition \ref{P:NTcom} ensures that it suffices to show that
\[
a\in\bigg(\bigcap_{\un{n}\perp F}X_\un{n}^{-1}(J_F(\L_\mt,X))\bigg)\cap\L_{\inv,F}\cap\L_{\lim,F}.
\]
Accordingly, fix $\un{n}=(n_1,\dots,n_d)\perp F$.
If $n_i>0$, then we argue as in (\ref{Eq:tenstrick}) to obtain that $a\in X_\un{n}^{-1}(J_F(\L_\mt,X))$.
Now suppose that $n_i=0$, so that $\un{n}\perp F\cup\{i\}$.
Applying invariance of $\L$ in tandem with the fact that $a\in\L_{F\cup\{i\}}$, we obtain that
\begin{equation*}
\sca{X_\un{n},aX_\un{n}}\subseteq\L_{F\cup\{i\}}\subseteq J_{F\cup\{i\}}(\L_\mt,X),
\end{equation*}
using item (i) of Definition \ref{D:NT tuple} in the final inclusion.
Note also that
\begin{align}\label{Eq:commtrick}
\sca{X_\un{i},\sca{X_\un{n},aX_\un{n}}X_\un{i}} & \subseteq\sca{X_\un{n}\otimes_A X_\un{i}, a(X_\un{n}\otimes_A X_\un{i})}\subseteq\sca{X_{\un{n}+\un{i}},aX_{\un{n}+\un{i}}} \\
											   & =\sca{X_{\un{i}+\un{n}},aX_{\un{i}+\un{n}}}=\sca{X_\un{i}\otimes_A X_\un{n},a(X_\un{i}\otimes_A X_\un{n})}\nonumber \\ 
											   & \subseteq[\sca{X_\un{n},\sca{X_\un{i},aX_\un{i}}X_\un{n}}]\subseteq\L_F\subseteq J_F(\L_\mt,X). \nonumber
\end{align}
Combining the previous two deductions, we have that 
\[
\sca{X_\un{n},aX_\un{n}}\subseteq X_\un{i}^{-1}(J_F(\L_\mt,X))\cap J_{F\cup\{i\}}(\L_\mt,X).
\]
Applying $[ \hspace{1pt} \cdot \hspace{1pt} ]_{\L_\mt}$ and invoking Proposition \ref{P:quo ideal}, we obtain that
\[
\sca{[X_\un{n}]_{\L_\mt}, [a]_{\L_\mt}[X_\un{n}]_{\L_\mt}}\subseteq[X_\un{i}]_{\L_\mt}^{-1}(\J_F([X]_{\L_\mt}))\cap\J_{F\cup\{i\}}([X]_{\L_\mt}).
\]
It now follows by Proposition \ref{P:Jpassdown} that
\[
[\sca{X_\un{n}, aX_\un{n}}]_{\L_\mt}=\sca{[X_\un{n}]_{\L_\mt}, [a]_{\L_\mt}[X_\un{n}]_{\L_\mt}}\subseteq\J_F([X]_{\L_\mt})
\]
and hence $a\in X_\un{n}^{-1}(J_F(\L_\mt,X))$ via another application of Proposition \ref{P:quo ideal}.
In total, we deduce that
\[
a\in\bigcap_{\un{n}\perp F}X_\un{n}^{-1}(J_F(\L_\mt,X)).
\]
To see that $a\in\L_{\inv,F}\equiv\bigcap_{\un{n}\perp F}X_\un{n}^{-1}(\cap_{F\subsetneq D}\L_D)$, fix $\un{n}=(n_1,\dots,n_d)\perp F$.
If $n_i>0$, then arguing as in (\ref{Eq:tenstrick}) gives that
\[
\sca{X_\un{n},aX_\un{n}}\subseteq\L_F\subseteq\cap_{F\subsetneq D}\L_D,
\]
where the final inclusion follows from the partial ordering of $\L$.
If $n_i=0$, then $\un{n}\perp F\cup\{i\}$ and so $\sca{X_\un{n},aX_\un{n}}\subseteq\L_{F\cup\{i\}}$ by invariance of $\L$.
Fix $D\supsetneq F$ and suppose that $i\in D$.
Then $F\cup\{i\}\subseteq D$ and so $\sca{X_\un{n},aX_\un{n}}\subseteq\L_D$ by the partial ordering of $\L$.
Now suppose that $i\not\in D$.
Observe that $|F|<|D|$ and so $|D|=d-n$ for some $1\leq n\leq N$.
By the inductive hypothesis, we have that
\[
\L_D=X_\un{i}^{-1}(\L_D)\cap\L_{D\cup\{i\}}.
\]
Note that $\sca{X_\un{n},aX_\un{n}}\subseteq\L_{D\cup\{i\}}$ by the partial ordering of $\L$.
By arguing as in (\ref{Eq:commtrick}) until the final inclusion, at which point we use that $\L_F\subseteq\L_D$ by the partial ordering of $\L$, we deduce that $\sca{X_\un{n},aX_\un{n}}\subseteq X_\un{i}^{-1}(\L_D)$.
Combining the preceding deductions, we obtain that $\sca{X_\un{n},aX_\un{n}}\subseteq\L_D$.
Since our choice of $D\supsetneq F$ was arbitrary, we ascertain that $\sca{X_\un{n},aX_\un{n}}\subseteq\cap_{F\subsetneq D}\L_D$ in all cases.
In total, we have that $a\in\L_{\inv,F}$.

Finally, since $a\in X_\un{i}^{-1}(\L_F)$ and $X$ is proper, we apply (\ref{Eq: comp}) to obtain that $\phi_\un{i}(a)\in\K(X_\un{i}\L_F)$.
Proposition \ref{P:lim perp} then gives that $a\in\L_{\lim,F}$.
Combining the preceding deductions, we ascertain that $a\in\L_F$ and we conclude that
\[
\L_F=X_\un{i}^{-1}(\L_F)\cap\L_{F\cup\{i\}}.
\]
By strong downward induction, the proof is complete.
\end{proof}

\subsection{T-families are NT-$2^d$-tuples}\label{Ss:TimpliesNT}

Proving that all T-families are NT-$2^d$-tuples requires more work.
The strategy is as follows: given a T-family $\L$, we must show that it satisfies conditions (i)-(iii) of Definition \ref{D:NT tuple}, as well as the simplified version of condition (iv) prescribed by Proposition \ref{P:NTcom}.
We start by showing that $\L$ satisfies conditions (i)-(iii), but do so out of order since we will need conditions (ii) and (iii) to obtain condition (i).

\begin{proposition}\label{P:Tfaminv}
Let $X$ be a proper product system over $\bZ_+^d$ with coefficients in a C*-algebra $A$ and let $\L$ be a T-family of $X$.
Then $\L$ is $X$-invariant.
\end{proposition}
\begin{proof}
Fix $F\subseteq[d]$ and $\un{n}\perp F$.
Since $\L$ consists of ideals, it suffices to show that
\[
\sca{X_\un{n},\L_FX_\un{n}}\subseteq\L_F.
\]
Without loss of generality, we may assume that $F\neq[d]$ and $\un{n}\neq\un{0}$. 
We proceed by induction on $|\un{n}|$.
First suppose that $|\un{n}|=1$, so that $\un{n}=\un{i}$ for some $i\in[d]\setminus F$.
Since $\L$ is a T-family, we have that $\L_F\subseteq X_\un{i}^{-1}(\L_F)$ and hence $\sca{X_\un{i},\L_FX_\un{i}}\subseteq\L_F$, as required.
Now suppose that we have proved the claim for all $\un{n}\perp F$ satisfying $|\un{n}|=N$ for some $N\in\bN$.
Fix $\un{n}\perp F$ such that $|\un{n}|=N+1$.
Then we may write $\un{n}=\un{m}+\un{i}$ for some $\un{m}\perp F$ satisfying $|\un{m}|=N$ and some $i\in[d]\setminus F$.
We obtain that
\[
\sca{X_\un{n},\L_FX_\un{n}}=\sca{X_\un{m}\otimes_A X_\un{i},\L_F(X_\un{m}\otimes_A X_\un{i})}\subseteq[\sca{X_\un{i},\sca{X_\un{m},\L_FX_\un{m}}X_\un{i}}]\subseteq\L_F,
\]
where we appeal to the inductive hypothesis in tandem with the base case in the final inclusion.
Thus, by induction, we conclude that $\L$ is invariant.
\end{proof}

\begin{proposition}\label{P:Tfampo}
Let $X$ be a proper product system over $\bZ_+^d$ with coefficients in a C*-algebra $A$ and let $\L$ be a T-family of $X$.
Then $\L$ is partially ordered.
\end{proposition}
\begin{proof}
Fixing $F\subseteq D\subseteq[d]$, we must show that $\L_F\subseteq\L_D$.
This is immediate when $F=D$, so assume that $F\subsetneq D$.
By relabelling elements if necessary, we may assume that $F=[k]$ and $D=[\ell]$ for some $0\leq k<\ell\leq d$ (with the convention that if $k=0$ then $F=\mt$).
Since $k+1\not\in F$ and $\L$ is a T-family, we have that $\L_F\subseteq\L_{F\cup\{k+1\}}$.
Likewise, since $k+2\not\in F\cup\{k+1\}$ and $\L$ is a T-family, we have that $\L_{F\cup\{k+1\}}\subseteq\L_{F\cup\{k+1,k+2\}}$.
Arguing iteratively in this way until $D\setminus F$ has been exhausted, we obtain the sequence of inclusions
\[
\L_F\subseteq\L_{F\cup\{k+1\}}\subseteq\L_{F\cup\{k+1,k+2\}}\subseteq\dots\subseteq\L_{F\cup(D\setminus F)}\equiv\L_D.
\]
Thus $\L_F\subseteq\L_D$ and we conclude that $\L$ is partially ordered.
\end{proof}

\begin{proposition}\label{P:Tfamcond1}
Let $X$ be a proper product system over $\bZ_+^d$ with coefficients in a C*-algebra $A$ and let $\L$ be a T-family of $X$.
Then $\L_F\subseteq J_F(\L_\mt,X)$ for all $\mt\neq F\subseteq[d]$.
\end{proposition}
\begin{proof}
Fix $\mt\neq F\subseteq[d]$ and $a\in\L_F$.
It suffices to show that $aX_F^{-1}(\L_\mt)\subseteq\L_\mt$ by properness of $X$.
Note that $\L_\mt$ is positively invariant by Proposition \ref{P:Tfaminv} and thus $X_F^{-1}(\L_\mt)=\bigcap_{i\in F}X_\un{i}^{-1}(\L_\mt)$ by \cite[Lemma 4.1.2]{DeK24}.
Fix $b\in X_F^{-1}(\L_\mt)$.
By relabelling elements if necessary, we may assume that $F=[k]$ for some $1\leq k\leq d$.
We start by setting 
\[
F_1:=F\setminus\{k\}.
\]
Since $a\in\L_F$ and $\L_F$ is an ideal, we have that $ab\in\L_F\equiv\L_{F_1\cup\{k\}}$.
Additionally, we have that $b\in X_F^{-1}(\L_\mt)\subseteq X_\un{k}^{-1}(\L_\mt)$ and hence
\[
ab\in X_\un{k}^{-1}(\L_\mt)\subseteq X_\un{k}^{-1}(\L_{F_1}),
\]
where the membership follows since $X_\un{k}^{-1}(\L_\mt)$ is an ideal and the inclusion follows by the comments preceding Theorem \ref{T:Kat par} together with the fact that $\L$ is partially ordered by Proposition \ref{P:Tfampo}.
Hence $ab\in X_\un{k}^{-1}(\L_{F_1})\cap\L_{F_1\cup\{k\}}$ and so $ab\in\L_{F_1}$ since $\L$ is a T-family.
Next we set 
\[
F_2:=F_1\setminus\{k-1\}=F\setminus\{k-1,k\}.
\]
We have that $ab\in\L_{F_1}\equiv\L_{F_2\cup\{k-1\}}$ and $b\in X_F^{-1}(\L_\mt)\subseteq X_\un{k-1}^{-1}(\L_\mt)$, so that
\[
ab\in X_\un{k-1}^{-1}(\L_\mt)\subseteq X_\un{k-1}^{-1}(\L_{F_2}),
\]
arguing as in the previous step.
Hence $ab\in X_\un{k-1}^{-1}(\L_{F_2})\cap\L_{F_2\cup\{k-1\}}$ and so $ab\in\L_{F_2}$ since $\L$ is a T-family.
We iterate the preceding argument until all elements of $F$ have been exhausted, yielding that $ab\in\L_\mt$.
Thus $aX_F^{-1}(\L_\mt)\subseteq\L_\mt$, finishing the proof.
\end{proof}

Ascertaining whether or not T-families satisfy the condition stated in Proposition \ref{P:NTcom} (via a direct argument) was an open question in the author's PhD thesis \cite{De24}.
Therein, an affirmative answer was only obtained in the setting of row-finite $k$-graphs \cite[Remark 6.2.8]{De24}.
Here we present a fully general affirmative answer, though it will require some set-up.
We begin with a lemma that holds on the level of general strong compactly aligned product systems.

\begin{lemma}\label{L:insideinv}
Let $X$ be a strong compactly aligned product system with coefficients in a C*-algebra $A$ and let $\L$ be a $2^d$-tuple of $X$ that is invariant, partially ordered and consists of ideals.
Fixing $\mt\neq F\subsetneq[d]$, we have that $\L_F\subseteq\L_{\inv,F}$ and thus $[\L_{\inv,F}]_{\L_F}$ is an ideal of $[A]_{\L_F}$.
\end{lemma}
\begin{proof}
Showing that $\L_F\subseteq\L_{\inv,F}\equiv\bigcap_{\un{n}\perp F}X_\un{n}^{-1}(\cap_{F\subsetneq D}\L_D)$ amounts to proving that
\[
\sca{X_\un{n},\L_FX_\un{n}}\subseteq\cap_{F\subsetneq D}\L_D\foral\un{n}\perp F.
\]
To this end, fix $\un{n}\perp F$.
By invariance of $\L$, we have that $\sca{X_\un{n},\L_FX_\un{n}}\subseteq\L_F$.
It follows from the partial ordering of $\L$ that $\L_F\subseteq\cap_{F\subsetneq D}\L_D$.
Combining these deductions, we obtain that $\L_F\subseteq\L_{\inv,F}$, as required.
Since both $\L_F$ and $\L_{\inv,F}$ are ideals of $A$, the second claim follows immediately and the proof is complete.
\end{proof}

Let $X$ be strong compactly aligned and let $\L$ be a $2^d$-tuple of $X$ that is invariant and consists of ideals. 
Fixing $F\subseteq[d]$ and $\un{n}\perp F$, the invariance condition implies that $\L_F$ is positively invariant for $X_\un{n}$.
Hence, appealing to the quotient construction outlined in Subsection \ref{Ss:C*-cor}, we deduce that $[X_\un{n}]_{\L_F}$ is a C*-correspondence over $[A]_{\L_F}$ with left action
\[
[\phi_\un{n}]_{\L_F}\colon[A]_{\L_F}\to\L([X_\un{n}]_{\L_F}); [a]_{\L_F}\mapsto[\phi_\un{n}(a)]_{\L_F}\foral a\in A.
\]
Note that we can replace ``$\L([X_\un{n}]_{\L_F})$" by ``$\K([X_\un{n}]_{\L_F})$" when $X$ is proper by Lemma \ref{L:Kat07}.
We emphasise that we can only guarantee this C*-correspondence structure on $[X_\un{n}]_{\L_F}$ when $\un{n}\perp F$.

\begin{proposition}\label{P:restinj}
Let $X$ be a proper product system over $\bZ_+^d$ with coefficients in a C*-algebra $A$ and let $\L$ be a T-family of $X$.
Then, fixing $\mt\neq F\subsetneq[d]$ and $\un{n}\perp F$, we have that the restriction of $[\phi_\un{n}]_{\L_F}$ to $[\L_{\inv,F}]_{\L_F}$ is injective and therefore isometric.
\end{proposition}
\begin{proof}
Since $\L$ is invariant by Proposition \ref{P:Tfaminv}, we can make sense of $[\phi_\un{n}]_{\L_F}$ by the comments preceding the statement.
In addition, as $\L$ is partially ordered by Proposition \ref{P:Tfampo}, we can make sense of the ideal $[\L_{\inv,F}]_{\L_F}$ of $[A]_{\L_F}$ by Lemma \ref{L:insideinv}.
Since $[\L_{\inv,F}]_{\L_F}$ is an ideal and thus in particular a C*-subalgebra of $[A]_{\L_F}$, the restriction of $[\phi_\un{n}]_{\L_F}$ to $[\L_{\inv,F}]_{\L_F}$ is a $*$-homomorphism between C*-algebras.
This justifies the final assertion of the statement, as any injective $*$-homo\-morphism between C*-algebras is automatically isometric.

We proceed now to the proof.
We begin by providing a characterisation of membership to $\ker[\phi_\un{n}]_{\L_F}$ which will be useful to us going forward.
More precisely, fixing $a\in A$, we have that
\begin{align}\label{Eq:kerchar}
[a]_{\L_F}\in\ker[\phi_\un{n}]_{\L_F} & \iff [\phi_\un{n}(a)\xi_\un{n}]_{\L_F}=0\foral \xi_\un{n}\in X_\un{n}\iff aX_\un{n}\subseteq X_\un{n}\L_F \\
					       & \iff \sca{X_\un{n},aX_\un{n}}\subseteq\L_F\iff a\in X_\un{n}^{-1}(\L_F), \nonumber
\end{align}
where the third equivalence follows by \cite[Proposition 1.3]{Kat07}.

We will prove that the restriction of $[\phi_\un{n}]_{\L_F}$ to $[\L_{\inv,F}]_{\L_F}$ is injective by induction on $|\un{n}|$.
First suppose that $|\un{n}|=0$ and therefore $\un{n}=\un{0}$.
Taking $[a]_{\L_F}\in\ker[\phi_\un{0}]_{\L_F}\cap[\L_{\inv,F}]_{\L_F}$, an application of (\ref{Eq:kerchar}) yields that $a\in A^{-1}(\L_F)=\L_F$ and hence $[a]_{\L_F}=0$.
Thus $[\phi_\un{0}]_{\L_F}$ is injective on $[\L_{\inv,F}]_{\L_F}$, as required. 

Now suppose that $|\un{n}|=1$.
Since $\un{n}\perp F$, it follows that $\un{n}=\un{i}$ for some $i\in[d]\setminus F$.
Taking $[a]_{\L_F}\in\ker[\phi_\un{i}]_{\L_F}\cap[\L_{\inv,F}]_{\L_F}$, an application of (\ref{Eq:kerchar}) yields that $a\in X_\un{i}^{-1}(\L_F)$.
By assumption we also have that $a\in\L_{\inv,F}\equiv\bigcap_{\un{m}\perp F}X_\un{m}^{-1}(\cap_{F\subsetneq D}\L_D)$, and so in particular 
\[
a\in\cap_{F\subsetneq D}\L_D\subseteq\L_{F\cup\{i\}}.
\]
Combining the preceding deductions, we obtain that
\[
a\in X_\un{i}^{-1}(\L_F)\cap\L_{F\cup\{i\}}=\L_F,
\]
where the equality follows from the fact that $\L$ is a T-family.
Thus $[a]_{\L_F}=0$ and so $[\phi_\un{i}]_{\L_F}$ is injective on $[\L_{\inv,F}]_{\L_F}$, as required.

Next, suppose that we have shown that $[\phi_\un{n}]_{\L_F}$ is injective on $[\L_{\inv,F}]_{\L_F}$ for all $\un{n}\perp F$ such that $|\un{n}|=N$ for some fixed $N\in\bN$.
Take $\un{n}\perp F$ such that $|\un{n}|=N+1$.
Then we may write $\un{n}$ in the form $\un{n}=\un{m}+\un{i}$ for some $\un{m}\perp F$ satisfying $|\un{m}|=N$ and some $i\in[d]\setminus F$.
Taking $[a]_{\L_F}\in\ker[\phi_\un{n}]_{\L_F}\cap[\L_{\inv,F}]_{\L_F}$, an application of (\ref{Eq:kerchar}) yields that $a\in X_\un{n}^{-1}(\L_F)$.
Additionally, we have that
\begin{align*}
[\sca{X_\un{n},aX_\un{n}}]=[\sca{X_{\un{m}+\un{i}},aX_{\un{m}+\un{i}}}]=[\sca{X_\un{m}\otimes_A X_\un{i}, a(X_\un{m}\otimes_A X_\un{i})}]=[\sca{X_\un{i},\sca{X_\un{m},aX_\un{m}}X_\un{i}}].
\end{align*}
Combining the previous two deductions, we obtain that
\[
\sca{X_\un{i},\sca{X_\un{m},aX_\un{m}}X_\un{i}}\subseteq\L_F.
\]
By definition, this means that $\sca{X_\un{m},aX_\un{m}}\subseteq X_\un{i}^{-1}(\L_F)$.
By assumption we also have that $a\in\L_{\inv,F}$ and so in particular
\[
\sca{X_\un{m},aX_\un{m}}\subseteq\cap_{F\subsetneq D}\L_D\subseteq\L_{F\cup\{i\}}.
\]
Thus we have that
\[
\sca{X_\un{m},aX_\un{m}}\subseteq X_\un{i}^{-1}(\L_F)\cap\L_{F\cup\{i\}}=\L_F,
\]
where the equality follows from the fact that $\L$ is a T-family.
By definition, this means that $a\in X_\un{m}^{-1}(\L_F)$.
An application of (\ref{Eq:kerchar}), replacing ``$\un{n}$" by ``$\un{m}$", yields that $[a]_{\L_F}\in\ker[\phi_\un{m}]_{\L_F}$.
Since in addition we have that $[a]_{\L_F}\in[\L_{\inv,F}]_{\L_F}$, we may appeal to the inductive hypothesis to deduce that $[a]_{\L_F}=0$. 
Hence $[\phi_\un{n}]_{\L_F}$ is injective on $[\L_{\inv,F}]_{\L_F}$, as required.
By induction, this completes the proof.
\end{proof}

We are now ready to show that all T-families are NT-$2^d$-tuples.

\begin{proposition}\label{P:TimpliesNT}
Let $X$ be a proper product system over $\bZ_+^d$ with coefficients in a C*-algebra $A$.
Then every T-family of $X$ is an NT-$2^d$-tuple of $X$.
\end{proposition}
\begin{proof}
Let $\L$ be a T-family of $X$.
First note that $\L$ satisfies conditions (i), (ii) and (iii) of Definition \ref{D:NT tuple} by Propositions \ref{P:Tfamcond1}, \ref{P:Tfaminv} and \ref{P:Tfampo}, respectively.
Thus, appealing to Proposition \ref{P:NTcom}, it suffices to show that
\[
\bigg(\bigcap_{\un{n}\perp F}X_\un{n}^{-1}(J_F(\L_\mt,X))\bigg)\cap\L_{\inv,F}\cap\L_{\lim,F}\subseteq\L_F\foral\mt\neq F\subsetneq[d].
\]
To this end, fix $\mt\neq F\subsetneq[d]$ and an element
\[
a\in\bigg(\bigcap_{\un{n}\perp F}X_\un{n}^{-1}(J_F(\L_\mt,X))\bigg)\cap\L_{\inv,F}\cap\L_{\lim,F}.
\]
Fixing $\varepsilon>0$, the fact that $a\in\L_{\lim,F}$ ensures that there exists $\un{n}\perp F$ such that
\[
\|\phi_\un{n}(a)+\K(X_\un{n}\L_F)\|<\varepsilon.
\]
In turn, we have that
\[
\|[\phi_\un{n}]_{\L_F}([a]_{\L_F})\|=\|[\phi_\un{n}(a)]_{\L_F}\|=\|\phi_\un{n}(a)+\K(X_\un{n}\L_F)\|<\varepsilon,
\]
where the first equality follows from the discussion preceding Proposition \ref{P:restinj}.
The second equality follows via a combination of Lemma \ref{L:Kat07} and the First Isomorphism Theorem for C*-algebras; more precisely, the mapping
\[
\K(X_\un{n})/\K(X_\un{n}\L_F)\to\K([X_\un{n}]_{\L_F}); k_\un{n}+\K(X_\un{n}\L_F)\mapsto[k_\un{n}]_{\L_F}\foral k_\un{n}\in\K(X_\un{n})
\]
is a $*$-isomorphism.
By assumption we also have that $a\in\L_{\inv,F}$ and hence $[a]_{\L_F}\in[\L_{\inv,F}]_{\L_F}$.
An application of Proposition \ref{P:restinj} then gives that
\[
\|[a]_{\L_F}\|=\|[\phi_\un{n}]_{\L_F}([a]_{\L_F})\|<\varepsilon.
\]
Since $\varepsilon>0$ is arbitrary, it follows that $[a]_{\L_F}=0$ and hence $a\in\L_F$, finishing the proof.
\end{proof}

With this, we arrive at the main result of the current work.

\begin{theorem}\label{T:NT=T}
Let $X$ be a proper product system over $\bZ_+^d$ with coefficients in a C*-algebra $A$.
Then the following hold:
\begin{enumerate}
\item the NT-$2^d$-tuples of $X$ are exactly the T-families of $X$, and
\item the NO-$2^d$-tuples of $X$ are exactly the O-families of $X$.
\end{enumerate}
\end{theorem}
\begin{proof}
Item (i) follows either by a combination of Propositions \ref{P:maxNT} and \ref{P:Tfamga} (the indirect route), or Propositions \ref{P:NTimpliesT} and \ref{P:TimpliesNT} (the direct route).
Item (ii) follows immediately by item (i) and Definitions \ref{D:rel NO tuple} and \ref{D:T+O}.
\end{proof}

\section{Applications}\label{S:app}

\noindent We conclude by applying Theorem \ref{T:NT=T} to give a simplification of Theorem \ref{T:NO param} in the proper case.
We then interpret the parametrising objects in the cases of C*-dynamical systems and row-finite higher-rank graphs.

\subsection{Gauge-invariant ideal structure of $\N\O(\K,X)$}\label{Ss:relCNPparam}

\noindent We begin by reinterpreting $\K$-relative NO-$2^d$-tuples (see Definition \ref{D:rel NO tuple}) in the proper case using the T-family machinery.

\begin{corollary}\label{C:relOfam}
Let $X$ be a proper product system over $\bZ_+^d$ with coefficients in a C*-algebra $A$ and let $\K$ be a $2^d$-tuple of $X$.
Then the $\K$-relative NO-$2^d$-tuples of $X$ are exactly the T-families $\L$ of $X$ that satisfy $\K\subseteq\L$.
\end{corollary}
\begin{proof}
The result follows immediately by Theorem \ref{T:NT=T} and Definition \ref{D:rel NO tuple}.
\end{proof}

Pursuant to Corollary \ref{C:relOfam}, we give the T-families therein their own name.

\begin{definition}\label{D:relOfam}
Let $X$ be a proper product system over $\bZ_+^d$ with coefficients in a C*-algebra $A$ and let $\K$ and $\L$ be $2^d$-tuples of $X$.
We say that $\L$ is a \emph{$\K$-relative O-family (of $X$)} if $\L$ is a T-family of $X$ that satisfies $\K\subseteq\L$.
\end{definition}

We now use Definition \ref{D:relOfam} to recast Theorem \ref{T:NO param} in the proper setting.

\begin{theorem}\label{T:NOparamprop}
Let $X$ be a proper product system over $\bZ_+^d$ with coefficients in a C*-algebra $A$ and let $\K$ be a $2^d$-tuple of $X$.
Then there exists an order-preserving bijection between the set of $\K$-relative O-families of $X$ and the set of gauge-invariant ideals of $\N\O(\K,X)$.
\end{theorem}
\begin{proof}
The result follows immediately by Theorem \ref{T:NO param} and Corollary \ref{C:relOfam}.
\end{proof}

The bijection of Theorem \ref{T:NOparamprop} is bolstered to a lattice isomorphism by equipping the set of $\K$-relative O-families with the lattice structure on the set of $\K$-relative NO-$2^d$-tuples outlined in Subsection \ref{Ss:NT} (this is justified by Corollary \ref{C:relOfam}).
Theorem \ref{T:NOparamprop} completely describes the gauge-invariant ideal structure of every equivariant quotient of $\N\T_X$, since every such quotient is canonically $*$-isomorphic to a relative Cuntz-Nica-Pimsner algebra (not necessarily of $X$, but certainly of a quotient of $X$, which is proper by Lemma \ref{L:Kat07}) \cite[Proposition 4.2.1]{DeK24}.
Note that $\N\O_X$ falls within this remit.

The main advantage of Theorem \ref{T:NOparamprop} compared to Theorem \ref{T:NO param} lies in the simpler description of the parametrising objects.
However, this simplification comes at a slight loss of generality in the passage from strong compactly aligned product systems to proper ones.
Nevertheless, Theorem \ref{T:NOparamprop} accounts for an array of important examples, including regular product systems over $\bZ_+^d$, product systems arising from C*-dynamical systems and row-finite higher-rank graphs, and product systems over $\bZ_+^d$ whose fibres (apart from the coefficient algebra) admit finite frames.
The reader is directed to \cite[Chapter 5]{De24} or \cite[Section 5]{DeK24} for further details on these examples.

\subsection{C*-dynamical systems}\label{Ss:c*-dynsys}

\noindent We now seek to interpret the parametrising objects of Theorem \ref{T:NOparamprop} in the setting of C*-dynamical systems.
We present the minimum amount of theory that will be needed; the reader is directed to \cite{DFK17, De24, DeK24, DK18} for further details.
In particular, full proofs of the assertions featuring in this subsection can be found in \cite[Section 5.3]{De24}.

A \emph{C*-dynamical system} $(A,\al,\bZ_+^d)$ consists of a C*-algebra $A$ and a unital semigroup homomorphism $\al\colon\bZ_+^d\to\End(A)$.
Fixing a C*-dynamical system $(A,\al,\bZ_+^d)$, we set
\[
X_{\al,\un{n}}:=[\al_\un{n}(A)A]\foral\un{n}\in\bZ_+^d.
\]
Note that each $X_{\al,\un{n}}$ inherits the usual right Hilbert $A$-module structure from $A$ and can be endowed with the structure of a C*-correspondence over $A$ via the left action
\[
\phi_\un{n}\colon A\to\L(X_{\al,\un{n}}); \phi_\un{n}(a)\xi_\un{n}=\al_\un{n}(a)\xi_\un{n}\foral a\in A \; \text{and} \; \xi_\un{n}\in X_{\al,\un{n}}.
\]
Additionally, we have that each $X_{\al,\un{n}}$ is proper.
The collection $X_\al:=\{X_{\al,\un{n}}\}_{\un{n}\in\bZ_+^d}$ then carries the structure of a proper product system over $\bZ_+^d$ with coefficients in $A$, where the multiplication maps are given by
\[
X_{\al,\un{n}}\otimes_A X_{\al,\un{m}}\to X_{\al,\un{n}+\un{m}}; \xi_\un{n}\otimes\xi_\un{m}\mapsto\al_\un{m}(\xi_\un{n})\xi_\un{m}\foral\xi_\un{n}\in X_{\al,\un{n}}, \xi_\un{m}\in X_{\al,\un{m}} \; \text{and} \; \un{n},\un{m}\in\bZ_+^d.
\]
C*-dynamical systems can be studied through the lens of the associated product systems, with concepts from the theory of the former having analogues in the theory of the latter (and vice versa).
Indeed, fixing $\un{n}\in\bZ_+^d$ and an ideal $I\subseteq A$, we have that
\begin{equation}\label{Eq:kerpredynsys}
\ker\phi_\un{n}=\ker\al_\un{n}\qand X_{\al,\un{n}}^{-1}(I)=\al_\un{n}^{-1}(I).
\end{equation}
Applying (\ref{Eq:kerpredynsys}) to recast the definitions of $\J$ and $\I$ in the language of C*-dynamical systems, we obtain that
\begin{equation}\label{Eq:dynsysJandI}
\J_F=(\bigcap_{i\in F}\ker\al_\un{i})^\perp 
\qand 
\I_F=\bigcap_{\un{n}\perp F}\al_\un{n}^{-1}((\bigcap_{i\in F}\ker\al_\un{i})^\perp)\foral\mt\neq F\subseteq[d].
\end{equation}

The following two remarks address points raised in Subsections \ref{Ss:T fam} and \ref{Ss:NTimpliesT}, respectively.

\begin{remark}\label{R:GIUTdiff}
Let $X$ be a proper product system over $\bZ_+^d$ with coefficients in a C*-algebra $A$.
Here we will show that the set of T-families of $X$ and the set of $2^d$-tuples $\L$ of $X$ that satisfy $\L\subseteq\I$ are not comparable, in general.

Firstly, let $B$ be a non-zero C*-algebra and let $A=B\oplus\bC$ be its unitisation.
We define a semigroup action $\al\colon\bZ_+^2\to\text{End}(A)$ by
\[
\al_{(m,n)}(b,\la)=\begin{cases} (0,\la) & \text{if} \; n\geq 1, \\ (b,\la) & \text{otherwise}, \end{cases}
\]
for all $(b,\la)\in A$.
In turn, the triple $(A,\al,\bZ_+^2)$ constitutes a C*-dynamical system and thus we obtain a proper product system $X_\al$ over $\bZ_+^2$ with coefficients in $A$.
Using (\ref{Eq:dynsysJandI}), it is routine to check that the $2^2$-tuple $\I$ decomposes as follows:
\begin{center}
\begin{tikzcd}
\I_{\{2\}}\arrow[dash]{r}\arrow[dash]{d} & \I_{\{1,2\}}=A\arrow[dash]{d} \\
\I_\mt=\{0\}\arrow[dash]{r} & \I_{\{1\}}=A.
\end{tikzcd}
\end{center}
We define a $2^2$-tuple $\L$ of $X_\al$ by
\begin{center}
\begin{tikzcd}
\L_{\{2\}}=\{0\} \arrow[dash]{r}\arrow[dash]{d} & \L_{\{1,2\}}=B\oplus\{0\}\arrow[dash]{d} \\
\L_\mt=\{0\}\arrow[dash]{r} & \L_{\{1\}}=\{0\}.
\end{tikzcd}
\end{center}
Observe that $\L\subseteq\I$.
However, \cite[Example 3.1.12]{DeK24} illustrates that $\L$ is not maximal and therefore not an NT-$2^2$-tuple by Proposition \ref{P:maxNT}.
In turn, the $2^2$-tuple $\L$ is not a T-family by Theorem \ref{T:NT=T}.

Next, let $X$ be a regular product system over $\bZ_+^d$ with coefficients in a non-zero simple C*-algebra $A$.
Then the family $\{\L_F\}_{F\subseteq[d]}$ where $\L_F=A$ for all $F\subseteq[d]$ is an NT-$2^d$-tuple of $X$ by \cite[Proposition 5.2.5]{DeK24}.
Hence it is a T-family by Theorem \ref{T:NT=T}.
However, this $2^d$-tuple is not contained in $\I$ (because $A\not\subseteq\I_\mt=\{0\}$).
\end{remark}

\begin{remark}\label{R:Jnotinv}
Let $X$ be a proper product system over $\bZ_+^d$ with coefficients in a C*-algebra $A$.
Here we will show that the $2^d$-tuple $\J$ of $X$ is not a T-family, in general.
This follows because the reverse inclusion in the statement of Proposition \ref{P:Jpassdown} may not hold, since $\J$ may not be invariant and so $\J_F\not\subseteq X_\un{i}^{-1}(\J_F)$.
Let us provide a counterexample to this effect.

Let $B$ be a non-zero C*-algebra and set $A=B\oplus B$.
We define a $*$-endomorphism $\al$ via
\[
\al\colon A\to A; (b,b')\mapsto (0,b)\foral(b,b')\in A.
\]
Note that $\ker\al=\{0\}\oplus B$.
By setting $\al_{(1,0)}:=\al$ and $\al_{(0,1)}:=\al$, we obtain a unital semigroup homomorphism $\bZ_+^2\to\text{End}(A)$ which we also denote by $\al$.
Thus $(A,\al,\bZ_+^2)$ is a C*-dynamical system and so induces a proper product system $X_\al$.
We have that
\[
\J_{\{1\}}=(\ker\al_{(1,0)})^\perp=B\oplus\{0\},
\]
using (\ref{Eq:dynsysJandI}) in the first equality.
Fixing $b\in B\setminus\{0\}$, we have that $(b,0)\in\J_{\{1\}}$ but also that
\[
\al_{(0,1)}(b,0)=(0,b)\not\in\J_{\{1\}}.
\]
Hence $\J_{\{1\}}\not\subseteq X_{\al,(0,1)}^{-1}(\J_{\{1\}})$, recalling that $X_{\al,(0,1)}^{-1}(\J_{\{1\}})=\al_{(0,1)}^{-1}(\J_{\{1\}})$ by (\ref{Eq:kerpredynsys}).
\end{remark}

The following result simplifies \cite[Corollary 5.3.5]{DeK24}.

\begin{corollary}\label{C:dynsysinterp}
Let $(A,\al,\bZ_+^d)$ be a C*-dynamical system and let $\K$ and $\L$ be $2^d$-tuples of $X_\al$.
Then $\L$ is a $\K$-relative O-family of $X_\al$ if and only if the following three conditions hold:
\begin{enumerate}
\item $\L$ consists of ideals,
\item $\L_F=\al_\un{i}^{-1}(\L_F)\cap\L_{F\cup\{i\}}\foral F\subsetneq[d] \; \text{and} \; i\in[d]\setminus F$, and
\item $\K\subseteq\L$.
\end{enumerate}
\end{corollary}
\begin{proof}
The result follows immediately by translating Definition \ref{D:relOfam} into the language of C*-dynamical systems using (\ref{Eq:kerpredynsys}).
\end{proof}

\subsection{Higher-rank graphs}\label{Ss:hrgraph}

Finally, we will interpret the parametrising objects of Theorem \ref{T:NOparamprop} in the setting of row-finite higher-rank graphs.
We present the minimum amount of theory that will be needed; the reader is directed to \cite{De24, RS05, RSY03, RSY04, Sim06} for further details.
In particular, full proofs of the assertions featuring in this subsection can be found in \cite[Section 5.4]{De24}.
For the remainder of the discussion, we will reserve $d$ for the degree map of a graph $(\La,d)$ of rank $k$.

Fix $k\in\bN$.
A \emph{$k$-graph} $(\La,d)$ consists of a countable small category $\La=(\text{Obj}(\La),\text{Mor}(\La),r,s)$ together with a functor $d\colon\La\to\bZ_+^k$, called the \emph{degree map}, satisfying the \emph{factorisation property}: 
\begin{quote}
For all $\la\in\text{Mor}(\La)$ and $\un{m},\un{n}\in\bZ_+^k$ such that $d(\la)=\un{m}+\un{n}$, there exist unique $\mu,\nu\in\text{Mor}(\La)$ such that $d(\mu)=\un{m}, d(\nu)=\un{n}$ and $\la=\mu\nu$.
\end{quote}
Here we view $\bZ_+^k$ as a category consisting of a single object, and whose morphisms are exactly the elements of $\bZ_+^k$ (when viewed as a set).
Composition in this category is given by entrywise addition, and the identity morphism is $\un{0}$.
Therefore, $d$ being a functor means that 
\[
d(\la\mu)=d(\la)+d(\mu)\; \text{and} \; d(\id_v)=\un{0} \foral\la,\mu\in\text{Mor}(\La) \; \text{satisfying} \; r(\mu)=s(\la) \; \text{and} \; v\in\text{Obj}(\La).
\]
We view $k$-graphs as generalised graphs, and therefore refer to the elements of $\text{Obj}(\La)$ as \emph{vertices} and the elements of $\text{Mor}(\La)$ as \emph{paths}.
Fixing $\la\in\text{Mor}(\La)$, the factorisation property guarantees that $d(\la)=\un{0}$ if and only if $\la=\id_{s(\la)}$.
Hence we may identify $\text{Obj}(\La)$ with the set
\[
\{\la\in\text{Mor}(\La)\mid d(\la)=\un{0}\}, 
\]
and consequently we may write $\la\in\La$ instead of $\la\in\text{Mor}(\La)$ without any ambiguity.

Fix a $k$-graph $(\La,d)$.
Given $\la\in\La$ and $E\subseteq\La$, we define
\[
\la E:=\{\la\mu\in\La\mid \mu\in E, r(\mu)=s(\la)\}\qand E\la:=\{\mu\la\in\La\mid \mu\in E, r(\la)=s(\mu)\}.
\]
In particular, we may replace $\la$ by a vertex $v\in\La$ and write
\[
vE:=\{\la\in E\mid r(\la)=v\}\qand Ev:=\{\la\in E\mid s(\la)=v\}.
\]
Fixing $\un{n}\in\bZ_+^k$, we set
\begin{align*}
\La^\un{n} :=\{\la\in\La\mid d(\la)=\un{n}\}.
\end{align*}
We say that $(\La,d)$ is \emph{row-finite} if $|v\La^\un{n}|<\infty$ for all $v\in\La^\un{0}$ and $\un{n}\in\bZ_+^k$.

Every $k$-graph $(\La,d)$ is canonically associated with a product system $X(\La) := \{X_\un{n}(\La)\}_{\un{n}\in\bZ_+^k}$ over $\bZ_+^k$ with coefficients in the C*-algebra $c_0(\La^\un{0})$, where we view $\La^\un{0}$ as a discrete space.
Firstly, set $X_\un{0}(\La):=c_0(\La^\un{0})$, which we view as a C*-correspondence over itself in the usual way.
For each $v\in\La^\un{0}$, we write $\delta_v\in c_0(\La^\un{0})$ for the projection on $\{v\}$.
For every $\un{0} \neq \un{n} \in \bZ_+^k$, we consider the linear space $c_{00}(\La^\un{n})$ and write $\de_\la$ for its generators.
A right pre-Hilbert $c_0(\La^\un{0})$-module structure on $c_{00}(\La^\un{n})$ is given by
\[
\sca{\xi_\un{n}, \eta_\un{n}}(v) := \sum_{s(\lambda) = v}\overline{\xi_\un{n}(\lambda)}\eta_\un{n}(\lambda) \qand (\xi_\un{n}a)(\lambda) :=\xi_\un{n}(\lambda)a(s(\lambda)),
\]
for all $\xi_\un{n},\eta_\un{n}\in c_{00}(\La^\un{n}), a\in c_0(\La^\un{0}), v\in\La^\un{0}$ and $\la\in\La^\un{n}$.
We write $X_{\un{n}}(\La)$ for the right Hilbert C*-module completion of $c_{00}(\La^{\un{n}})$.
A left action $\phi_\un{n}$ of $c_0(\La^\un{0})$ on $X_\un{n}(\La)$ is induced by
\[
\phi_\un{n}(a)\colon c_{00}(\La^\un{n})\to c_{00}(\La^\un{n}); (\phi_\un{n}(a)\xi_\un{n})(\la)=a(r(\la))\xi_\un{n}(\la)\foral a\in c_0(\La^\un{0}), \xi_\un{n}\in c_{00}(\La^\un{n}), \la\in \La^\un{n},
\]
thereby imposing a C*-correspondence structure on $X_\un{n}(\La)$.
Fixing $\un{n},\un{m}\in\bZ_+^k$, we define a multiplication map $u_{\un{n},\un{m}}$ by
\[
u_{\un{n},\un{m}}\colon X_\un{n}(\La)\otimes_{c_0(\La^\un{0})}X_\un{m}(\La)\to X_{\un{n}+\un{m}}(\La); u_{\un{n},\un{m}}(\delta_\la\otimes\delta_\mu)=\begin{cases} \delta_{\la\mu} & \text{if} \; r(\mu)=s(\la), \\ 0 & \text{otherwise,} \end{cases}
\]
for all $\la\in\La^\un{n}$ and $\mu\in\La^\un{m}$, rendering $X(\La)$ a product system over $\bZ_+^k$ with coefficients in $c_0(\La^\un{0})$.
The structure of $(\La,d)$ can be studied via $X(\La)$ and vice versa.
We will not dwell on this point, instead contenting ourselves with noting that $(\La,d)$ is row-finite if and only if $X(\La)$ is proper (see, e.g., \cite[Proposition 5.4.5]{De24}).

We will use the duality between ideals of $c_0(\La^{\un{0}})$ and subsets of $\La^{\un{0}}$ given by the mutually inverse mappings
\begin{align*}
I & \mapsto H_I:=\{v\in\La^\un{0}\mid \delta_v\in I\}, \; \text{for all ideals} \; I\subseteq c_0(\La^\un{0}); \\
H & \mapsto I_H:=\ol{\spn}\{\delta_v\mid v\in H\}, \foral H \subseteq \La^\un{0}.
\end{align*}
Note that this duality implements a lattice isomorphism, and that $I_\mt=\{0\}$ and $I_{\La^\un{0}}=c_0(\La^\un{0})$.

Let $\L$ be a $2^k$-tuple of $X(\La)$ that consists of ideals.
For notational convenience, we set $H_{\L,F}:=H_{\L_F}$ for all $F\subseteq[k]$ and $H_\L:=\{H_{\L,F}\}_{F\subseteq[k]}$.
For an ideal $I \subseteq c_0(\La^\un{0})$, we have that
\begin{equation} \label{eq:n la}
H_{X_{\un{n}}(\La)^{-1}(I)}=\{v\in\La^\un{0}\mid s(v\La^\un{n})\subseteq H_I \}
\foral \un{n} \in \bZ_+^k.
\end{equation}

The following result simplifies the row-finite case of \cite[Corollary 5.4.14]{DeK24}.

\begin{corollary}\label{C:kgraphinterp}
Let $(\La,d)$ be a row-finite $k$-graph.
Let $\K$ and $\L$ be $2^k$-tuples of $X(\La)$ and suppose that $\K$ consists of ideals.
Then $\L$ is a $\K$-relative O-family of $X(\La)$ if and only if the following three conditions hold:
\begin{enumerate}
\item $\L$ consists of ideals,
\item $H_{\L,F}=\{v\in\La^\un{0}\mid s(v\La^\un{i})\subseteq H_{\L,F}\}\cap H_{\L,F\cup\{i\}}$ for all $F\subsetneq[k] \; \text{and} \; i\in[k]\setminus F$, and
\item $H_{\K,F}\subseteq H_{\L,F}$ for all $F\subseteq[k]$.
\end{enumerate}
\end{corollary}
\begin{proof}
The result follows immediately by translating Definition \ref{D:relOfam} into the language of higher-rank graphs.
This is accomplished by using the duality between ideals of $c_0(\La^\un{0})$ and subsets of $\La^\un{0}$, together with (\ref{eq:n la}).
\end{proof}

Corollary \ref{C:kgraphinterp}, employed in tandem with Theorem \ref{T:NOparamprop}, aligns with the first part of \cite[Theorem 5.5]{BB24}.
This can be seen by taking $\K=\{\{0\}\}_{F\subseteq[k]}$ for the T-family case and $\K=\I$ for the O-family case.
We need to stipulate that $\K$ consists of ideals in the statement of Corollary \ref{C:kgraphinterp} in order to exploit the duality between ideals of $c_0(\La^\un{0})$ and subsets of $\La^\un{0}$.
This is sufficient, since for a general $2^k$-tuple $\K$ of $X(\La)$, we have that $\N\O(\K,X(\La))=\N\O(\sca{\K},X(\La))$ by the comments preceding Definition \ref{D:LCNP}.


\end{document}